\documentclass[12pt]{amsart}
\pdfoutput=1
\usepackage{microtype}
\usepackage[active]{srcltx}
\usepackage{calc,amssymb,amsthm,amsmath,amscd, eucal,ulem, mathtools}
\usepackage{phaistos}
\usepackage{alltt}
\RequirePackage[dvipsnames,usenames]{xcolor}

\input{mabliautoref.sty}
\input{xy}
\xyoption{all}
\input{kmacros3.sty}
\usepackage{tikz,calc}
\usepackage{tikz-cd}
\usetikzlibrary{external}
\usepackage[marvosym]{tikzsymbols}
\usepackage{amsfonts, mathrsfs}
\usepackage[cal=boondox, calscaled=1.05]{mathalfa}
\usepackage{calligra}
\usepackage{stmaryrd} %power series bracket
\usepackage{dcpic, pictexwd}
\usepackage[left=1in,top=1in,right=1in,bottom=1in]{geometry}
\usepackage{bm}
\usepackage{verbatim}
\usepackage{upgreek}
\usepackage{hyperref}
\numberwithin{equation}{theorem}

\newcommand{\F}{\mathbb{F}}

\renewcommand{\:}{\colon}
\newcommand{\eg}{{\itshape e.g.} }
\renewcommand{\m}{\mathfrak{m}}
\renewcommand{\n}{\mathfrak{n}}
\newcommand{\p}{\mathfrak{p}}
\newcommand{\kay}{\mathcal{k}}
\newcommand{\el}{\mathcal{l}}

\newcommand{\q}{\mathfrak{q}}

\DeclareMathOperator{\Ass}{Ass}
\DeclareMathOperator{\GL}{GL}

\DeclareMathOperator{\Der}{Der}

\let\Ass\relax
\DeclareMathOperator{\Ass}{\mathbf{Ass}}
\DeclareMathOperator{\Assoc}{Ass}

\DeclareMathOperator{\ssHom}{ \sH\!\!\;\!\!\text{\calligra{\Large om}\,}}

\usepackage{setspace}
\usepackage{enumerate}
\usepackage{graphicx}
\usepackage[all,cmtip]{xy}

\usepackage{verbatim}

\renewcommand{\sA}{\mathcal{A}}

\renewcommand{\sC}{\mathcal{C}}

\renewcommand{\sF}{\mathcal{F}}

\renewcommand{\sH}{\mathcal{H}}

\renewcommand{\sO}{\mathcal{O}}

\newcommand{\eps}{\varepsilon}
\usepackage{marvosym}

\begin{document}
\title{Pulling back Cartier structures along regular maps} 
\author[J.~Carvajal-Rojas]{Javier Carvajal-Rojas}
\address{Centro de Investigaci\'on en Matem\'aticas, A.C., Callej\'on Jalisco s/n, 36024 Col. Valenciana, Guanajuato, Gto, M\'exico}
\email{\href{mailto:javier.carvajal@cimat.mx}{javier.carvajal@cimat.mx}}
\author[A.~St\"abler]{Axel St\"abler}
\address{Universit\"at Leipzig\\ Mathematisches Institut\\ Augustusplatz 10\\
04109 Leipzig\\Germany} 
\email{\href{mailto:staebler@math.uni-leipzig.de}{staebler@math.uni-leipzig.de}}

\keywords{Cartier module, regular morphism, twisted inverse image, (mixed) test ideals}

\thanks{Carvajal-Rojas was partially supported by the grants ERC-STG \#804334, FWO \#G079218N, and SECIHTI \#CBF2023-2024-224, \#CF-2023-G-33, and \#CBF-2025-I-673. Stäbler was supported in part by SFB/TRR 45 Bonn-Essen-Mainz funded by DFG}

\subjclass[2020]{13A35, 14B25}

\begin{abstract}
We introduce a framework for pulling back Cartier modules and their associated invariants along regular $F$-finite morphisms. To achieve this, we construct a relative Cartier isomorphism and operator for an arbitrary regular $F$-finite map of locally noetherian schemes. As an application, we obtain new results on the constancy regions of mixed test ideals, based on the work of Felipe P\'erez.
\end{abstract}
\maketitle

\section{Introduction}

Let $f: X \to Y$ be a smooth morphism between locally noetherian schemes; that is, $f$ is a formally smooth morphism of finite type, or equivalently, a regular morphism of finite type. In this situation, the relative cotangent sheaf $\Omega_f$ is a locally free coherent sheaf on $X$, and we can form the associated canonical sheaf $\omega_f \coloneqq \det \Omega_f$. This allows us to define the \emph{twisted/exceptional inverse image} functor on coherent sheaves by
\[ f^! \coloneqq \omega_f \otimes f^* .\]
This functor is fundamental in the framework of Serre duality and, more broadly, in the abstract theory of Grothendieck duality.

On the other hand, the theory of \emph{$F$-singularities} in positive characteristic is substantially based on the Grothendieck duality of the Frobenius morphisms. For example, this framework is used to define the \emph{Cartier operator} $\kappa \: F_* \omega_X \to \omega_X$ on a smooth (and then normal) scheme $X$ over a perfect field of characteristic $p>0$; see \cite[Theorem 7.2]{KatzNilpotentConnections} and \cite{BlickleSchwedeSurveyPMinusE}.

It is therefore not surprising that the theory of Cartier modules (and crystals) is compatible with $f^!$. This was worked out by Blickle and the second named author in \cite{BlickleStablerFunctorialTestModules}, where they proved that $f^!$ can be viewed as a functor between Cartier modules. More importantly, their construction commutes with \emph{funtorial test modules}; which we briefly recall next. 

In judging the ($F$-)singularities of a Cartier module $(\sF,\phi)$, one considers the complexity and size of its lattice of Cartier submodules. The funtorial test module $\uptau(\sF,\phi) \subset \sF$ is the smallest \emph{nontrivial} one; where it is a bit technical and non-obvious to define non-triviality in this generality (see \autoref{def.testmodulesummary}). Thus, the test module gives a lower bound on how singular the Cartier module is. For example, $(\sF,\phi)$ is said to be \emph{$F$-regular} if $\uptau(\sF,\phi) = \sF$. In the case of $(\omega_X,\kappa)$ and $X$ Cohen--Macaulay, this refers to the $F$-rationality of $X$, which is the $F$-singularity analog to rational singularities.

All of this is very pleasing, but in many applications there is a problem: it depends on $f$ being of finite type. We will mention an explicit example of this being an issue in the context of constancy regions of mixed test ideals. Our key observation here is that, in positive characteristics, there is another more natural notion of finiteness that we can use instead. Namely, the \emph{$F$-finiteness} of $f$, which also ensures that $\Omega_f$ is coherent. To explain what this is, let us consider the following cartesian diagram.
\[
\xymatrixcolsep{5pc}\xymatrixrowsep{3pc}\xymatrix{
X \ar@/_/[ddr]_-{f} \ar@/^/[drr]^-{F^e_X} \ar@{.>}[dr]|-{F^e_{f}}\\
&X^{(q)} \ar[d]_-{f^{(q)}} \ar[r]^-{G^e_f}  & X\ar[d]^-{f} \\
&Y \ar[r]^-{F^e_Y} &Y}
\]
The map $F_f^e \: X \to X^{(q)}$ is the \emph{$e$-th Frobenius morphism} of a given morphism $f$. Following Hashimoto \cite{HashimotoFfinitenessAndItsdescent}, we say that $f$ is $F$-finite if $F^e_f$ is finite for all $e>0$ (it suffices to consider $e=1$ \cite[Definition 2.3]{CarvajalRojasStablerPristineMorphisms}). 

As we shall explain, if $f \: X \to Y$ is a regular $F$-finite morphism of locally noetherian schemes, then $\Omega_f$ is a locally free coherent sheaf and we can define $\omega_f \coloneqq \det \Omega_f$. Furthermore, there is nothing stopping us from defining the functor $f^! \coloneqq \omega_f \otimes f^*$ from coherent sheaves on $Y$ to coherent sheaves on $X$. We are able to construct meaningful foundational results for the $F$-singularity theory in this setup. The first is as follows, for which we highlight that $f^{(q)}$ is another regular $F$-finite map of locally noetherian schemes.
\begin{theoremA*}[{\autoref{thm.regulardualizing}, \autoref{Sec.RelCartierOperartor}}] \label{theoA}
Let $f\:X \to Y$ be a regular $F$-finite map of locally noetherian schemes. With notation as above, there is a \emph{relative Cartier isomorphism} 
\[
\alpha_e \:\omega_f^{1-q} \xrightarrow{\sim} F^{e,!}_f \sO_{X^{(q)}},
\] 
where $F_f^{e,!}$ is Grothendieck duality functor for finite covers applied to $F_f^e$. Moreover, there is an induced \emph{relative Cartier operator} natural transformation 
\[
\varkappa^e \: F^e_{f \ast} f^! \to f^{(q)!}.
\]
\end{theoremA*}

For $Y = \mathbb{F}_p$, i.e., the absolute case, \autoref{thm.regulardualizing} was also recently and independently obtained by Daichi Takeuchi in \cite[Proposition A.3]{DaichiBernsteinSatoDModulePosChar} employing similar methods.

Our construction bypasses all finite type conditions and relies instead on the local equivalence between a differential basis of $\Omega_f$ and that of a $p$-basis of $f\: X \to Y$. This equivalence is due to Tyc \cite{TycDifferentialBasisAndp-basis} (see also the work of Ma and Polstra \cite[Section 11]{polstramafsingularities} for a great exposition and complete proofs). Our relative Cartier isomorphism and operators are conceptualizations of these equivalences and of the way the corresponding change of bases matrices relate to each other. Even though our formulation might look abstract and formal, it is actually still quite down-to-earth. To illustrate this, we also obtain an explicit description of these maps in the local setting (see \autoref{le.explicitrelativecartierstructureregular}). These descriptions further show that our maps coincide with the classical one in the finite-type setup. 

As an application of this result, if $\sC$ is a Cartier $\sO_Y$-algebra, then we construct a functor $f^! \sF \coloneqq f^\ast \sF \otimes \omega_f$ from Cartier $\sC$-modules to Cartier $f^\ast \sC$-modules. Of course, these more general pairs $(f^*,f^!)$ generalize the aforementioned pullback pair defined in \cite{BlickleStablerFunctorialTestModules} for smooth morphisms. However, it should be noted that our definition is different from theirs (we cannot use any finite-type hypothesis). Furthermore, we obtain the following extensions of their results (here, $\uptau$ is the test module functor introduced in \cite{BlickleStablerFunctorialTestModules}).

\begin{theoremB*}[\autoref{cor.tauregulartransforms}]
\label{theoB}
Let $f \: X \to Y$ be a regular $F$-finite morphism of locally noetherian schemes. Then, there is a natural isomorphism $\uptau \circ f^{!} \to f^{!} \circ \uptau$ of functors.
\end{theoremB*}
To test our theory, we generalize foundational results by P\'erez about the shape of mixed test ideals \cite{perezconstancyregions}. Pérez's results are about $\bF_q$-algebras of finite type, and his proofs rely on the base field being finite. It has not been possible yet to extend his proofs beyond this case although it is expected/conjectured that his results hold over any $F$-finite base field. Our formalism allows to ``pull back'' and so extend his results to suitable field extensions of finite fields. A special case is the following (in \autoref{sec.constancyregions} we have weaker assumptions on $\kay$).

\begin{theoremC*}[\autoref{theo.finitenessconstacyregions}, \autoref{theo.pfractal}]
Let $\kay$ be a finite extension of $\F_q(t_1, \ldots, t_k)$ and $R$ be a finite type $\kay$-algebra with ideals $\mathfrak{a}_1, \ldots, \mathfrak{a}_n$. Let $(M,\sC)$ be a Cartier module with $\sC$ a finitely generated Cartier algebra. For any $T \in \bR_{\geq 0}$,  the set
\[ \{\uptau(M, \sC, \mathfrak{a}_1^{t_1}, \ldots, \mathfrak{a}_n^{t_n}) \mid  0 \leq t_i \leq T\}\] is finite. Moreover, given $\bm{s} = (s_1, \ldots, s_n) \in \mathbb{R}^n$, the set
\[ \{ \bm{t} \in \mathbb{R}^n \mid \uptau(M, \sC, \mathfrak{a}_1^{t_1}, \ldots, \mathfrak{a}_n^{t_n}) = \uptau(M, \sC, \mathfrak{a}_1^{s_1}, \ldots, \mathfrak{a}_n^{s_n}) \} \]
is of the form $\varphi^{-1}(i)$ for some $i \in \mathbb{N}$, where $\varphi$ is a $p$-fractal function.
\end{theoremC*}

\subsection*{Acknowledgements} 
 The authors would like to thank Manuel~Blickle, Rankeya~Datta, Anne~Fayolle, Daniel~Fink,  Steffen~Fr\"ohlich, Sri Iyengar, Peter~McDonald, Daniel~Smolkin, and Maciej~Zdanowicz for very useful discussions in the preparation of this article.

\section{On Regular Morphisms}\label{sec.pristineandregular}

In the local study of morphisms of locally noetherian schemes, the archetype of what constitutes a good property requires a ``continuous variation'' on the fibers (e.g., flatness) and ``nonsingular'' fibers in some sense. The quintessential example is \emph{regularity.} A morphism of locally noetherian schemes is said to be \emph{regular} if it is flat and has geometrically regular fibers.  It turns out that this notion is encoded in the differential structure of the morphism and, more precisely, in its cotangent complex. To make this precise, let us consider the affine case given by a homomorphism $\theta \:R\to S$ of noetherian rings with cotangent complex $\mathbb{L}_{\theta}$. This reduction is justified by \cite[\href{https://stacks.math.columbia.edu/tag/07R8}{Tag 07R8}]{stacks-project}.

Then $\theta$ is regular if and only if its first André--Quillen homology with coefficients in $S$
\[
H_1(\theta,S) \coloneqq \mathbf{H}_1(\bL_{\theta}),
\]
vanishes and its zeroth one (i.e., the module of Kähler differentials) 
\[
H_0(\theta,S) \coloneqq \mathbf{H}_0(\bL_{\theta}) = \Omega_{\theta}
\] 
is flat as an $S$-module. See \cite[Proposition 5.6.2]{Majadas2010}.

Fortunately, there is a more down-to-earth way to express these homological conditions, namely the so-called \emph{jacobian criterion for regularity}.  It says that $\theta$ is regular if and only if for every (equivalently some) presentation
\[
0 \to I \to A \coloneqq R[x_i \mid i\in I] \to S \to 0
\]
of $\theta$ (as the quotient of a polynomial $A$-algebra) the canonical sequence of $S$-modules
\begin{equation} \label{eqn.IIFundamentalExactSequence}
    0 \to I/I^2 \xrightarrow{\delta} S \otimes_A \Omega_{A/R} \to \Omega_{\theta} \to 0
\end{equation}
is exact (i.e., $\delta$ is injective) and $\Omega_{\theta}$ is flat. See \cite[Propositions 1.4.6, 5.6.2]{Majadas2010}.

This shows that regular morphisms are closely related to formally smooth ones (in the noetherian case). Indeed, with $\theta$ as above, we have the following two equivalent ways to think about its formal smoothness (see \cite[Theorem 2.3.1]{Majadas2010}):
\begin{enumerate}
    \item $\theta$ is regular and $\Omega_{\theta}$ is a projective $B$-module.
    \item (Jacobian criterion for formal smoothness) For every $A$-algebra presentation $B=S/I$ as above, the homomorphism $\delta \: I/I^2 \to S \otimes_A \Omega_{A/R}$ admits a $B$-linear retraction (i.e., the sequence \autoref{eqn.IIFundamentalExactSequence} is split exact). 
\end{enumerate}

In particular, for \emph{kählerian morphisms} of locally noetherian schemes (i.e.\,those for which the cotangent sheaf is coherent), regularity and formal smoothness are the exact same notion. These are exactly the morphisms we are going to work with here.  

Observe that essentially of finite type morphisms of locally noetherian schemes are kählerian. However, in characteristic zero, they might be hard to find beyond this case. For example, $\bC[x] \to \bC \llbracket x\rrbracket$ is not kählerian. Nevertheless, $\bF_p[x] \to \bF_p \llbracket x\rrbracket$ is kählerian and, in fact, kählerian morphisms are much more abundant in positive characteristics. To understand why, we need to bring to scene the \emph{Frobenius morphisms} of a map.
  
Let $f \: X \to Y$ be a morphism of locally noetherian schemes. Its \emph{$e$-th (relative) Frobenius morphism} is the map $F_f^e \: X \to X^{(q)}$ defined by the following cartesian diagram.
\begin{equation}\label{relFrobenius} \xymatrixcolsep{5pc}\xymatrixrowsep{3pc}\xymatrix{
X \ar@/_/[ddr]_-{f} \ar@/^/[drr]^-{F^e_X} \ar@{.>}[dr]|-{F^e_{f}}\\
&X^{(q)} \ar[d]_-{f^{(q)}} \ar[r]^-{G^e_f}  & X\ar[d]^-{f} \\
&Y \ar[r]^-{F^e_Y} &Y}
\end{equation}
In the affine case $f = \Spec \theta \: \Spec S \to \Spec R$, we write
\[
F_{\theta}^e = F^e_{S/R} \: S_{F^e_* R} \coloneqq S \otimes_R F^e_* R \to F^e_* S, \quad s \otimes F^e_*r\mapsto F^e_* s^qr.
\]
For further details on Frobenius morphisms, we invite the reader to see \cite[\S2]
{CarvajalRojasStablerPristineMorphisms}. However, there is one key principle that we will use several times, and so it is worth isolating. 

\begin{remark}[Iterability of Frobenius] \label{rem.IterabilityOfFrobenius}
   As explained in \cite[p. 4]{CarvajalRojasStablerPristineMorphisms}, we have that
    \begin{equation} \label{eqn.Iterability1}
         F_f^{e+1}=F_{f^{(q)}} \circ \cdots \circ F_{f^{(p)}} \circ F_f = F_{f^{(q)}} \circ F_f^e
    \end{equation}
    and more generally
 \begin{equation} \label{eqn.Iterability2}
       F^{e+e'}_f = F^{e'}_{f^{(q)}} \circ F^e_f = F^{e}_{f^{(q')}} \circ F^{e'}_f.
    \end{equation}
Thus, in iterating the relative Frobenius, one has to use the Frobenius maps of $f^{(q)}$ accordingly. The same formulas hold with $G$ in place of $F$.
\end{remark}

The first key point is that 
$
\Omega_f = \Omega_{F^e_f}
$ and so $f$ is kählerian if $F^e_f$ is a finite map (for some/equivalently all $e \geq 1$). Following Hashimoto \cite{HashimotoCMFinjectiveHoms}, we say that $f$ is \emph{$F$-finite} if this is the case. This includes the case in which $F_f$ is an isomorphism, of which there are many maps, as studied in \cite{CarvajalRojasStablerPristineMorphisms}. Since we work with locally noetherian schemes, we call such maps \emph{pristine}. Moreover, if $X$ is $F$-finite then so is $f$, and the converse holds provided that $Y$ is $F$-finite. In particular, $X$ being $F$-finite would ensure that $f$ is kählerian. It should be noted that the use of $F$-finitenes for schemes is a mild hypothesis. It implies excellence, and in the local case, the converse is true provided that the residue field is $F$-finite (e.g. perfect) \cite[Theorem 10.5]{polstramafsingularities}.

Conversely, a result of Fogarty states that $f$ is kählerian if and only if it is $F$-finite; see \cite[Corollary 11.4]{polstramafsingularities}. These sorts of equivalences between differentials and Frobenius run deeper. Crucial for us is the relative Kunz theorem of Radu--André (\cite{RaduRelativeKunz,AndreRelativeKunz}, \cf \cite[\S10]{polstramafsingularities}). It states that $f$ is regular if and only if $F^e_f$ is flat for all $e>0$. Moreover, in either case $X^{(q)}$ is a locally noetherian scheme.

We will work with regular $F$-finite morphisms of locally noetherian schemes. We sum up their  key properties as follows.

\begin{proposition}
\label{ffiniteregularisfsmooth}
Let $f\colon X \to Y$ be a morphism of locally noetherian schemes. The following statements are equivalent.
\begin{enumerate}
    \item $f$ is a regular $F$-finite morphism,
    \item $f$ is a formally smooth map and $\Omega_f$ is a locally free coherent sheaf.
    \item $F_f^e \: X \to X^{(q)}$ is a finite flat map (of locally noetherian schemes).
\end{enumerate}
\end{proposition}

\begin{corollary}[Stability under base change]
\label{co.ffiniteregularbasechange}
Let $f\: X \to Y$ be a regular $F$-finite morphism of locally noetherian schemes. If $Z \to Y$ is a base change with $Z$ and $X \times_Y Z$ locally noetherian , then $f'\:X \times_Y Z \to Z$ is a regular $F$-finite morphism.
\end{corollary}
\begin{proof}
By \cite[Lemma 2 (5)]{HashimotoFfinitenessAndItsdescent}, $f'$ is $F$-finite. Thus, by \autoref{ffiniteregularisfsmooth}, regularity and formal smoothness coincide for $f'$, but the latter is stable under base change (\cite[\href{https://stacks.math.columbia.edu/tag/02H2}{Lemma 02H2}]{stacks-project}).
\end{proof}

\subsection{Differential bases vs $p$-bases} To conclude our remarks on regular $F$-finite morphisms of locally noetherian schemes, we recall the notion of differential basis and how it can be used to construct a simultaneous trivialization of the locally free coherent sheaves $\Omega_f$ and $F^e_{f*} \sO_X$. 

    Let $\theta \: R \to S$ be an $R$-algebra and $\Gamma = \{x_{\lambda} \mid \lambda \in \Lambda\} \subset S$ be an indexed subset of $S$. The set $\Gamma$ is a \emph{$p$-basis} for $\theta$ if the homomorphism of $S_{F_*R}$-algebras 
    \[
    S_{F_*R}[t_{\lambda} \mid \lambda \in 
    \Lambda]/\langle t_{\lambda}^p- x_{\lambda} \otimes 1 \mid \lambda \in \Lambda \rangle \xrightarrow{t_{\lambda} \mapsto F_* x_{\lambda}} F_* S
    \]
    is an isomorphism. Let $I$ denote the set of functions $\bm{i}\: \Lambda \to {[0,p-1] \cap \bN}$, which we denote as $\lambda \mapsto i_{\lambda}$, whose values are all zero except possibly for a finite number of elements in $\Lambda$. Also, let $\bm{x}^{\bm{i}} \coloneqq \prod_{\lambda \in \Lambda} x_{\lambda}^{i_{\lambda}}$ be the corresponding monomial. Then, $\Gamma$ is a $p$-basis if and only if $\{F_* \bm{x}^{\bm{i}} \mid \bm{i} \in I\}\subset F_* S$ is a basis for $F_*S$ as an $S_{F_*R}$-module. When $\Gamma=\{x_1,\ldots,x_n\}$ is a finite set, this can be expressed more succinctly as a direct sum decomposition.
    \[
    F_*S = \bigoplus_{0 \leq i_1,\ldots,i_n \leq p-1} S_{F_* R} F_* x_1^{i_1} \cdots x_n^{i_n}.
    \]
    Moreover, if we consider the corresponding dual basis $\{(F_* x_1^{i_1} \cdots x_n^{i_n})^{\vee}\}_{0 \leq i_1,\ldots,i_n \leq p-1}$, then 
    \[
    \Phi_{\theta} \coloneqq \Phi_{\theta,\Gamma} \coloneqq (F_* x_1^{p-1} \cdots x_n^{p-1})^{\vee}
    \]
    is a free generator of the $F_*S$-module $\Hom_{S_{F_*R}}(F_*S,S_{F_*R})$ as
    \[
    (F_* x_1^{i_1} \cdots x_n^{i_n})^{\vee} = \Phi_{\theta} \cdot F_* x_1^{p-1-i_1} \cdots x_n^{p-1-i_n}.
    \]
    
    On the other hand, one says that $\Gamma$ is a \emph{differential basis}, or a \emph{$d$-basis} for short, for $\theta$ if
    \[
    S^{\oplus \lambda} \xrightarrow{e_{\lambda} \mapsto \mathrm{d}x_{\lambda}} \Omega_{\theta}
    \]
is an isomorphism of $S$-modules. In other words, $\{\mathrm{d}x
_{\lambda} \mid {\lambda} \in \Lambda \} \subset \Omega_{\theta}$ is an $S$-basis. 

\begin{proposition}
\label{pro.fsmoothdbasisfactorization}
    Let $\theta\: R\to S$ be a formally smooth algebra. It admits a differential basis if and only if it factors as $R \to T \to S$
    where $\sigma \: R \to T$ is a polynomial $R$-algebra and $\rho\:T \to S$ is formally \'etale.
\end{proposition}
\begin{proof}
    Write $T=R[t_{\lambda} \mid {\lambda}\in \Lambda]$ and let the $R$-algebra homomorphism $\rho\:T\to S$ be given by $t_{\lambda} \mapsto x_{\lambda} \in S$. Since $\theta$ is formally smooth, $H_1(\theta,S)=0$, and we have the exact sequence
    \[
    0 \to H_1(\rho,S) \to S \otimes_T\Omega_{\sigma}=\bigoplus_{\lambda\in \Lambda} S \mathrm{d}t_{\lambda} \xrightarrow{\mathrm{d}t_{\lambda} \mapsto \mathrm{d}x_{\lambda}} \Omega_{\theta} \to \Omega_{\rho} \to 0.
    \]
    Then, $\Gamma=\{x_{\lambda} \mid {\lambda}\in \Lambda\} \subset S$ is a $d$-basis for $\theta$ if and only if $H_1(\rho,S)=0$ and $\Omega_{\rho}=0$, i.e., $\rho\: T \to S$ is formally étale.
\end{proof}

One readily verifies by direct calculation that a $p$-basis must be a differential basis. The converse is much more difficult but holds if $F_{\theta}$ is a homomorphism of noetherian rings, as established by Ma–Polstra in \cite[Theorem 11.5]{polstramafsingularities} after Tyc's work \cite{TycDifferentialBasisAndp-basis}. For future reference, we record this fundamental result as follows.

\begin{theorem}[Tyc theorem]
\label{pbasisdiffbasisequivalence}
Let $\theta \colon R \to S$ be a regular homomorphism of noetherian rings and $\Gamma\subset S$ as above. Then, $\Gamma$ is a $p$-basis if and only if it is a $d$-basis.
\end{theorem}
\begin{proof}
Since $\theta$ is regular, the ring $S_{F_\ast R}$ is noetherian and it may be identified with its image $R[S^p]$ in $S$. Now apply \cite[Theorem 11.5]{polstramafsingularities}.
\end{proof}

\begin{corollary} \label{cor.StandardCartierOperators}
    Let $\theta \colon R \to S$ be a regular homomorphism of noetherian rings that admits a $d/p$-basis $x_1,\ldots,x_n \in S$. Then we have a direct sum decomposition
    \[
     F^e_*S = \bigoplus_{0 \leq i_1,\ldots,i_n \leq q-1} S_{F^e_* R} F^e_* x_1^{i_1} \cdots x_n^{i_n}.
    \]
    and the $F^e_*S$-module $\Hom_{S_{F^e_*R}}(F^e_*S,S_{F^e_*R})$ is freely generated by
    \[
    \Phi^e_{\theta} \coloneqq \Phi^e_{\theta,\Gamma} \coloneqq (F^e_* x_1^{q-1} \cdots x_n^{q-1})^{\vee}.
    \]
\end{corollary}
\begin{proof}
    Since being a $d$-basis is invariant under base change, $x_1,\ldots, x_n \in S_{F^e_* R}$ is a $d/p$-basis for $f^{(q)}$ for all $e$. The rest then follows by using \autoref{eqn.Iterability1} with $f = \Spec \theta$.
\end{proof}

\begin{corollary}
\label{co.regularpbasisfactorization}
   With notation as in \autoref{cor.StandardCartierOperators},
   $\theta$ admits a factorization of $R$-algebras $R \to T=R[t_1,\ldots,t_n] \to S$ where $T \to S$ is pristine.
\end{corollary}

Regarding the existence of $d/p$-bases, we can say the following.

\begin{proposition}[{\cf \cite[Proposition 11.8]{polstramafsingularities}}]
\label{le.diffbasisexist}
Let $f\: X \to Y$ be a regular $F$-finite morphism of locally noetherian schemes. Then, there are affine open coverings $Y = \bigcup_k V_k$ and $f^{-1}V_k = \bigcup_{l} U_{kl}$ such that $\sO_Y(V_k) \to \sO_X(U_{kl})$ admits a $d/p$-basis for all $k,l$. 
\end{proposition}
\begin{proof}
Since $\Omega_f$ is coherent, we may assume that $f$ is the spectrum of a local homomorphism $\theta\: (R, \m, \kay) \to (S, \n, \el)$. Since $\Omega_{\theta}$ is generated by $\mathrm{d}x$ with $x \in S$ and is free of finite rank, Nakayama's lemma implies that there are $\mathrm{d}x_1, \ldots, \mathrm{d}x_n$ that form a basis of $\Omega_{\theta}$.
\end{proof}

\begin{remark}
The existence of a $d$-basis is \emph{not} equivalent to $\Omega_f$ being free; see \cite[p. 392]{TycDifferentialBasisAndp-basis}.
\end{remark}

With notation as in \autoref{le.diffbasisexist}, we can then trivialize $\Omega_f$ and $F_{f*} \sO_X$ simultaneously. Indeed, letting $\Gamma_{kl} \subset \sO_X(U_{kl})$ be the (finite) $d/p$-basis on $U_{kl}$, then
\[
\Omega_f|_{U_{kl}} = \bigoplus_{x\in \Gamma_{kl}} \sO_{U_{kl}} \mathrm{d}x
\qquad \text{and} \qquad
F_{f*} \sO_X|_{U_{kl}} = \bigoplus_{\bm{i} \in I_{kl}} \sO_{U^{(p)}_{kl}} F_* \bm{x}^{\bm{i}}.
\]
But what can be said about the relationship between the corresponding transition matrices? This issue lies at the core of our construction of a twisted inverse image for Cartier modules. We will give a partial answer; yet, a complete solution to the following problem remains open.

\begin{question}[Tyc theorem and change of basis]
    Let $\theta\: R \to S$ be a regular $F$-finite homomorphism of noetherian rings. Let $x_1,\ldots,x_n \in S$ and $y_1,\ldots,y_n \in S$ be two $d/p$-bases. Consider the jacobian matrix $J \in \GL_n(S)$ given by
    $
    \mathrm{d} y_i = \sum_{j=1}^n J_{ij} \mathrm{d} x_j.
    $
    Likewise, consider the ``Frobenius-theoretic jacobian'' matrix $\Xi \in \GL_{p^n}(S_{F_* R})$ defined by
    $
    F_* \bm{y}^{\bm{i}} = \bigoplus_{\bm{j}} \Xi_{\bm{i},\bm{j}} F_* \bm{x}^{\bm{j}}.
    $
    How do $J$ and $\Xi$ compare to each other?
\end{question}

\section{Twisted Inverse Image for Regular $F$-finite Morphisms}
\label{RegularPullbacks}
The goal of this section is to construct a regular inverse image of a Cartier module. Throughout this section, we set $f\: X \to Y$ as a regular $F$-finite morphism of locally noetherian schemes. The first input is the construction of a relative Cartier isomorphism, which will enable the construction of relative Cartier operators and subsequently the sought twisted inverse image of Cartier modules. Recall that the sheaf of differentials $\Omega_f$ is a locally free coherent sheaf; see \autoref{ffiniteregularisfsmooth}. We then set $\omega_f \coloneqq \det \Omega_f$. 

\subsection{Relative Cartier isomorphisms}
As mentioned above, to construct an inverse image functor, we rely on (compatible) relative Cartier isomorphisms 
\[
\alpha_e\: \omega_f^{1-q} \xrightarrow{\sim} F_f^{e,!} \sO_{X^{(q)}}
\] 
which we shall construct next. Here, $F_f^{e,!} \sO_{X^{(q)}} \coloneqq \ssHom(F_{f \ast}^e \sO_X, \sO_{X^{(q)}})$. 
More generally, $F_f^{e,!}$ is the exceptional inverse image of the Frobenius as a finite map—it is the right adjoint functor of $F^e_{f*}$. Since $F^e_f$ is a faithfully flat finite morphism, we have a natural isomorphism 
\[
F_f^{e,!} \sO_{X^{(q)}} \otimes F_f^{e*} \xrightarrow{\sim} F_f^{e,!}.
\]
To simplify notation, we may sometimes write this as an equality.
\begin{remark}
    The above is equivalent to the ordinary relative Cartier isomorphism in the smooth case (\cf \eg \cite[A.2]{EmertonKisinTheRiemannHilbertCorrespondence}, \cite[Section 3]{stablerunitftestmodules}) because, for $e =1$, we obtain 
    \[
    \omega_f \cong F_f^! \sO_{X^{(p)}} \otimes \omega_f^p \cong F_f^! \sO_{X^{(p)}} \otimes (G_f \circ F_f )^\ast \omega_f \cong  F_f^! \sO_{X^{(p)}} \otimes F_f^* G^*_f \omega_f  \cong F_f^! G_f^\ast \omega_f.
    \]
\end{remark}

We also want this family of canonical isomorphisms $\{\alpha_e\}_{e\geq 1}$ to be compatible with one another. For example, we want them to be constructed inductively starting from $\alpha\coloneqq \alpha_1$. To this end, we start by recalling a well-known result, which we demonstrate for completeness.

\begin{lemma}
\label{le.shriekisofiniteflat}
 Consider a cartesian square of schemes
\[
\xymatrix{
T' \ar[r]^-{a'} \ar[d]_-{b'}& T \ar[d]^-{b} \\
S'\ar[r]^-{a}& S
}
\]
where $b$ is finite flat. 
There is a natural isomorphism $a'^* \circ b^! \to b'^! \circ a^*$.
\end{lemma}
\begin{proof}
Since $b$ is a finite flat map, we have a natural isomorphism $b^! \sO_S \otimes b^* \to b^!$.  This induces the natural isomorphism $a'^*b^!\sO_S \otimes c^* \to a'^* \circ b^!$,
where $c$ is the canonical map $T' \to S$. Similarly, $b'$ is also finite and flat, so we have a natural isomorphism $b'^! \sO_{S'} \otimes b'^* \to b'^!$, which yields the natural isomorphism $b'^! \sO_{S'} \otimes c^* \to b'^! \circ a^*$.
Thus, it suffices to construct a canonical isomorphism $a'^*b^!\sO_S \to b'^! \sO_{S'}$. To this end, we may work locally with $S = \Spec A$, $T = \Spec B$, and $S' = \Spec A'$. Then, the canonical homomorphism
\[
(A' \otimes_A B) \otimes_{B} \Hom_A(B,A) = A' \otimes_A \Hom_A(B,A) \xrightarrow{\sim} \Hom_{A'}(A' \otimes_A B,B)
\]
is an isomorphism as $B$ is a finitely generated flat $A$-module.
\end{proof}

   Using \autoref{eqn.Iterability2}, we may write $F^{e+e'}_f = F^{e}_{f^{(q')}} \circ F^{e'}_f$ and so
 \[
   F_f^{(e+e'),!} \sO_{X^{(qq')}} = F^{e',!}_f F^{e,!}_{f^{(q')}} \sO_{X^{(q'q)}} = F^{e',!}_f F^{e,!}_{f^{(q')}} G_{f^{(q)}}^{e'*}\sO_{X^{(q)}}.
    \]
Now, applying \autoref{le.shriekisofiniteflat} to the cartesian diagram
\[
\xymatrixcolsep{5pc}\xymatrix{
X^{(q')} \ar[r]^-{G^{e'}_f} \ar[d]_-{F^e_{f^{(q')}}} & X \ar[d]^-{F^e_f} \\
X^{(qq')} \ar[r]^-{G^{e'}_{f^{(q)}}} & X^{(q)}
}
\]
yields a natural isomorphism
    \[
 G_{f}^{e'*} \circ F^{e,!}_f \xrightarrow{\sim} F^{e,!}_{f^{(q')}} \circ G_{f^{(q)}}^{e'*},
    \]
and, in particular, a canonical isomorphism
\[
 G_{f}^{e'*} F^{e,!}_f \sO_{X^{(q)}} \xrightarrow{\sim} F^{e,!}_{f^{(q')}}  G_{f^{(q)}}^{e'*} \sO_{X^{(q)}}.
\]
We can summarize everything so far as follows.
\begin{definition}
    Defining the endofunctor on coherent sheaves $\Theta_f^e \coloneqq F_f^{e,!} \circ G^{e*}_f$, denote the above canonical isomorphisms as
    \[
\beta_e^{e'} \: \Theta^{e'}_f F^{e,!}_{f} \sO_{X^{(q)}}\xrightarrow{\sim}  F_f^{(e+e'),!} \sO_{X^{(qq')}}.
\]
\end{definition}

On the other hand, since $F_f$ is finite flat, we have a canonical isomorphism
\begin{align*}
 \Theta_f^{e'} \omega_f^{1-q} = F_f^{e',!}G_f^{e'*} \omega_f^{1-q} \xleftarrow{\sim } F_f^{e',!} \sO_{X^{(p)}} \otimes F_f^{e'*} G_f^{e'*} \omega_f^{1-q} &= F_f^{e',!} \sO_{X^{(p)}} \otimes F_X^{e'*} \omega_f^{1-q}\\
 &=  F_f^{e',!} \sO_{X^{(p)}} \otimes \omega_f^{q'(1-q)}.
\end{align*}
For ease of notation, we may write this as an equality. In this way, 
an isomorphism $\alpha_{e'}\colon \omega_f^{1-q'} \to F_f^{e',!} \sO_{X^{(q')}}$ induces an isomorphism 
\[
\alpha_{e'} \otimes \omega_f^{q'(1-q)} \: \omega_f^{1-qq'} \xrightarrow{\sim} \Theta_f^{e'} \omega_f^{1-q}.
\]

Suppose that we have further constructed an isomorphism $\alpha_e \: \omega_f^{1-q} \to F^{e,!}_f \sO_{X^{(q)}}$. Then we can construct $\alpha_{e+e'}$ as the composition of isomorphisms
 \[
 \omega_{f}^{1-qq'} \xrightarrow{\alpha_{e'} \otimes \omega_f^{q'(1-q)}} \Theta_f^{e'} \omega_f^{1-q}  \xrightarrow{ \Theta_f^{e'} \alpha_e} \Theta_f^{e'} F_f^{e,!} \sO_{X^{(q)}} \xrightarrow{\beta_e^{e'}} F_f^{(e+1),!} \sO_{X^{(qq')}}. 
 \]
 In other words, up to a canonical isomorphism of the target, $\alpha_{e+e'}$ can be obtained from $\alpha_e$ and $\alpha_{e'}$ as the composition of $\Theta_f^{e'} \alpha_e $ with $\alpha_{e'} \otimes \omega_f^{q'(1-q)}$.

 We can employ the above to construct the isomorphisms $\alpha_e$ inductively. We may start by constructing $\alpha = \alpha_1 \: \omega_f^{1-p} \to F_f^! \sO_{X^{(p)}}$ and then define $\alpha_e$ following the recursive rule
 \[
 \alpha_{e+1} \coloneqq \beta_e^1 \circ \Theta_f^1 \alpha_e \circ \Big(\alpha \otimes \omega_f^{p(1-q)}\Big) = \beta_1^{e} \circ \Theta_f^{e} \alpha \circ \Big(\alpha_e \otimes \omega_f^{q(1-p)}\Big).
 \]

\begin{theorem}[Relative Cartier isomorphisms]
\label{thm.regulardualizing}
Let $f\: X \to Y$ be a regular $F$-finite morphism of locally noetherian schemes as above. There are canonical isomorphisms \begin{equation}\label{eqn.ChoiceCartierOperators}\alpha_e\: \omega_f^{1-q} \xrightarrow{\sim} F_f^{e,!} \sO_{X^{(q)}}
\end{equation}
 such that (with notation as in \autoref{cor.StandardCartierOperators})
\[
\alpha_e \:(\mathrm{d}x_1 \wedge \cdots \wedge \mathrm{d}x_n)^{\otimes (1-q)} \mapsto \Phi^e_{f,\{x_1,\ldots,x_n\}}
\]
on any affine chart $f \: \Spec S \to \Spec R$ where there is a $d/p$-basis $x_1,\ldots,x_n \in S$. Moreover, 
\[
\alpha_{e+e'} \coloneqq \beta_e^{e'} \circ \Theta_f^{e'} \alpha_e \circ \Big(\alpha_{e'} \otimes \omega_f^{q'(1-q)}\Big)
\]
for all $0\neq e,e'\in \bN$. 
In particular, this yields compatible natural isomorphisms
 \[
 \gamma_e \:\omega_f^{1-q} \otimes F^{e*}_f \xrightarrow{\sim} F^{e,!}_f.
 \]
\end{theorem}

\begin{proof}
As we have already explained, it suffices to construct $\alpha = \alpha_1$. By \autoref{le.diffbasisexist}, we have an open covering of $f$ by affine open charts $U \to V$ such that
\[
\bigoplus_{i=1}^n \sO_{U} \mathrm{d}x_{i} \xrightarrow{\sim} \Omega_{f}|_{U}
\]
for some $x_1,\ldots,x_n \in \sO_X(U)$. By \autoref{pbasisdiffbasisequivalence} and \autoref{cor.StandardCartierOperators}, we obtain a trivialization of $F_{f*} \sO_{X}$
\[
\bigoplus_{0 \leq i_1,\ldots,i_n \leq p-1} \sO_{U^{(p)}} F_{f*} x_{1}^{i_1} \cdots x_{n}^{i_n}  \xrightarrow{\sim} \big(F_{f*} \sO_{X}\big)\big|_{U^{(p)}},
\]
and further, a trivialization of  $F_f^{!} \sO_{X^{(p)}}$
\[
\sO_{U} \cdot \big(F_{f*} x_{1}^{p-1} \cdots x_{n}^{p-1}\big)^{\vee} \xrightarrow{\sim} \big(F_f^{!} \sO_{X^{(p)}}\big)\big|_{U}.
\]

We must prove that this local trivialization realizes $F_f^{!} \sO_{X^{(p)}}$ as $\omega_X^{1-p}$ canonically; i.e., it is compatible with change of $d/p$-basis. More precisely, write $V = \Spec R$, $U = \Spec S$, and  let $y_1, \ldots, y_n \in S$ be another $d/p$-basis. Consider the jacobian matrix $J \in \GL_n(S)$ given by
    \[
    \mathrm{d} y_i = \sum_{j=1}^n J_{ij} \mathrm{d} x_j.
    \]
In particular, $\mathrm{d}y_1 \wedge \cdots \wedge \mathrm{d}y_n = (\det J) \cdot \mathrm{d}x_1 \wedge \cdots \wedge \mathrm{d}x_n$. On the other hand,
\begin{equation}
\label{eq.basechange}
F_{f*} \bm{y}^{\bm{j}} = \sum_{\bm 0\leq \bm{i} \leq \bm{p-1}} \Xi_{\bm{i},\bm{j}} F_{f*} \bm{x}^{\bm{i}}, \quad \bm 0 \leq \bm{j} \leq \bm{p-1}
\end{equation}
for some uniquely determined $\Xi_{\bm{i},\bm{j}} \in S_{F_* R}$, such that $\Xi \in \GL_{p^n}(S_{F_* R})$. Here, we use multi-index notation and write $\bm n = (n, \ldots, n)$ for $n \in \N$. Moreover,
\[
\left(F_{f \ast} \bm{x^{p-1}}\right)^\vee = \left (F_{f \ast} \bm{y^{p-1}}\right)^\vee \cdot F_{f*}\xi
\]
for some uniquely determined $\xi \in S^{\times}$. We need to prove the identity
\[
\xi = (\det J)^{p-1}.
\]

\begin{claim}
\label{claim.tracedualbasechange}
With notation as above, we have
\[ 
\xi = \sum_{\bm 0 
\leq \bm j \leq \bm{p-1}} (-1)^n \prod_{k=1}^n \left(\sum_{l=1}^n J_{lk} \partial_{y_l}\right)^{p-1}\left(\bm{y^{j}}\right) \cdot \bm{y^{p-1 -j}}.\]
\end{claim}
\begin{proof}[Proof of the claim]
We have, using \autoref{eq.basechange},
\[\left(F_{f \ast}\bm x^{\bm{p-1}}\right)^\vee \left(F_{f \ast} \bm y^{\bm j} \right) = \left(F_{f \ast}\bm x^{\bm{p-1}}\right)^\vee \left(  \Xi_{\bm{ p-1}, \bm j} F_{f \ast}\bm x^{\bm{p-1}}\right) = F_f\left(\Xi_{\bm{p -1}, \bm j}\right).\]
It follows that \begin{equation}
    \label{eq.basechange2}
\left(F_{f \ast} \bm x^{\bm{p-1}}\right)^\vee = \left(F_{f, \ast} \bm y^{\bm{p-1}}\right)^\vee \cdot \sum_{\bm 0 
\leq \bm j \leq \bm{p-1}}\Xi_{\bm{p-1}, \bm j} F_{f \ast}\bm y^{\bm{p-1-j}}.
\end{equation}

Let $\{\partial_{x_1}, \ldots, \partial_{x_n}\}$ and by $\{\partial_{y_1}, \ldots, \partial_{y_n}\}$ be the dual bases to $\mathrm{d}\bm x$ and $\mathrm{d}\bm y$ in $\Omega_{S/R}^\vee = \Der_R(S)$. 
In particular, \[
\partial_{x_k} = \sum_{l=1}^n J_{lk} \partial_{y_l}.
\]
From \autoref{eq.basechange}, we deduce that
\[ \partial_{x_1}^{p-1} \cdots \partial_{x_n}^{p-1}\left( F_{f \ast} \bm y^{\bm j}\right) = (p-1)!^n F_f\left(\Xi_{\bm{p-1}, \bm j}\right) = (-1)^n F_f\left(\Xi_{\bm{p-1}, \bm j}\right)\]
Therefore
\[
F_f\left(\Xi_{\bm{p-1}, \bm j}\right) = (-1)^n \prod_{k=1}^n \left(\sum_{l=1}^n J_{lk} \partial_{y_l}\right)^{p-1}\left(F_{f \ast}\bm y^{\bm j}\right).
\]
Plugging this into \autoref{eq.basechange2} yields the claimed formula for $\xi$.
\end{proof}
Using the Leibniz rule multiple times, we can further rewrite the expression for $\xi$ as
\[
   \xi = \sum_{\bm 0 
\leq \bm j \leq \bm{p-1}} (-1)^n \sum_{\bm 0 \neq \bm{i} \leq \bm{p-1}} \xi_{\bm{i}}(J) \partial_{\bm{y}}^{\bm{i}} \left(\bm{y^j} \right) \cdot\bm{y^{p-1 -j}}\nonumber
\]
for some $\xi_{\bm{i}}(J) \in S$. Moreover,
\begin{align}
\label{eq.nu1}
\xi &=  (-1)^n \sum_{\bm 0 
\leq \bm j \leq \bm{p-1}}\sum_{\bm 0 \neq \bm{i} \leq \bm{p-1}} \xi_{\bm{i}}(J) \partial_{\bm{y}}^{\bm{i}}(\bm{y^{j}}) \cdot \bm{y^{p-1-j}}\\
&= (-1)^n \sum_{\bm 0 
\leq \bm j \leq \bm{p-1}}\sum_{\bm 0 \neq \bm{i} \leq \bm{j}} \xi_{\bm{i}}(J) \prod_{l=1}^n \prod_{k=0}^{i_l-1} (j_l -k) \bm{y^{p-1-i}}\nonumber\\
& = (-1)^n \sum_{\bm 0 
\neq \bm i \leq \bm{p-1}} \xi_{\bm{i}}(J)\sum_{\bm i \leq \bm{j} \leq \bm{p-1}}  \prod_{l=1}^n \prod_{k=0}^{i_l-1} (j_l -k) \bm{y^{p-1-i}}\nonumber\\
&= (-1)^n \sum_{\bm 0 
\neq \bm i \leq \bm{p-1}} \xi_{\bm{i}}(J) \sum_{u=1}^n \sum_{j_u = i_u}^{p-1} \prod_{l=1}^n \prod_{k=0}^{i_l-1} (j_l -k) \bm{y^{p-1-i}}\nonumber.
\end{align}

\begin{claim} The following equalities hold in $\bF_p$
    \[
     \sum_{j=i}^{p-1}  \prod_{k=0}^{i-1}(j-k) = \begin{cases}
         0 & \text{if } i=1,\ldots,p-2\\
         -1 & \text{if } i=p-1.
     \end{cases}
    \]
\end{claim}
\begin{proof}[Proof of the claim]
    Observe that 
     \[
     \sum_{j=i}^{p-1}  \prod_{k=0}^{i-1}(j-k) = \sum_{j=i}^{p-1} i! \binom{j}{i} = i! \sum_{j=i}^{p-1} \binom{j}{i} = i! \binom{p}{i+1}
    \]
    from which the result follows.
\end{proof}
With this, \autoref{eq.nu1} simplifies to
\[
\xi = (-1)^n \xi_{\bm{p-1}}(J) \cdot (-1)^n = \xi_{\bm{p-1}}(J).
\]

\begin{claim} \label{Claim.CombinatorialFormula}
We have 
\[
\xi_{\bm{p-1}}(J) = \sum_{\substack{a \: \{1,\ldots,n\} \times \{1,\ldots,n\} \to \{0,\ldots,p-1\} \\ \sum_{k} a_{kl} = p-1 = \sum_l a_{kl}}  } \frac{(p-1)!^n}{\prod_{k,l} a_{lk}!} \prod_{k,l} J_{lk}^{a_{lk}}.
\]
In particular, $\xi_{\bm{p-1}}(J) = \xi_{\bm{p-1}}(J^\top)$
\end{claim}
\begin{proof}[Proof of the claim]
Since $y_1, \ldots, y_n$ are a $p$-basis and the $\partial_{y_i}$ are $S_{F_*R}$-linear, we have that 
\[
\partial_{y_i} a \partial_{y_j} =  (a\partial_{y_i} -1) \partial_{y_j} = a \partial_{y_i} \partial_{y_j} - \partial_{y_j}\]
for all $a \in S$. Applying this reasoning multiple times, we see that
\[ \left(\sum_{l=1}^n J_{lk} \partial_{y_l}\right)^{p-1} = \sum_{a_{1k}+ \cdots +a_{nk}=p-1} \binom{p-1}{a_{1k},\ldots,a_{nk}} J_{1k}^{a_{1k}}\cdots J_{nk}^{a_{nk}} \partial_{y_1}^{a_{1k}} \cdots \partial_{y_n}^{a_{nk}} + g\]
where $g$ is a polynomial in the $\bm \partial_y$ of total degree $< p-1$.\footnote{The degree of $\bm{\partial_y^j}$ is $\bm j$ and its total degree is $j_1 + \cdots + j_n$.}  Then,
\begin{align*}
&\prod_{k=1}^n \left(\sum_{l=1}^n J_{lk} \partial_{y_l} \right)^{p-1} \\ 
={} &\prod_{k=1}^n \sum_{a_{1k}+ \cdots +a_{nk}=p-1} \binom{p-1}{a_{1k},\ldots,a_{nk}} J_{1k}^{a_{1k}}\cdots J_{nk}^{a_{nk}} \partial_{y_1}^{a_{1k}} \cdots \partial_{y_n}^{a_{nk}} + g\\
={} &\sum_{\substack{a \: \{1,\ldots,n\} \times \{1,\ldots,n\} \to \{0,\ldots,p-1\} \\ \sum_{k} a_{kl} = p-1 = \sum_l a_{kl}}  } \left( \prod_{k=1}^n \binom{p-1}{a_{1k},\ldots,a_{nk}} J_{1k}^{a_{1k}}\cdots J_{nk}^{a_{nk}}\right) \partial_{y_1}^{a_{11}+ \cdots + a_{1n}} \cdots \partial_{y_n}^{a_{n1}+\cdots +a_{nn}}+h
\end{align*}
where $h$ is a polynomial whose term of degree $\bm {p-1}$ is zero. Hence, we obtain
\begin{align*}
\xi_{\bm{p-1}}(J) &= \sum_{\substack{a \: \{1,\ldots,n\} \times \{1,\ldots,n\} \to \{0,\ldots,p-1\} \\ \sum_{k} a_{kl} = p-1 = \sum_l a_{kl}}  } \prod_{k=1}^n \binom{p-1}{a_{k1},\ldots,a_{kn}} J_{1k}^{a_{1k}}\cdots J_{nk}^{a_{nk}} \\
&= \sum_{\substack{a \: \{1,\ldots,n\} \times \{1,\ldots,n\} \to \{0,\ldots,p-1\} \\ \sum_{k} a_{kl} = p-1 = \sum_l a_{kl}}  } \frac{(p-1)!^n}{\prod_{k,l} a_{lk}!} \prod_{k,l} J_{lk}^{a_{lk}}.
\end{align*}
\end{proof}

Inspired by the above, we are going to think of $\xi_{\bm{p-1}}$ as an operator
\[
\xi_{\bm{p-1}}=\xi_{\bm{p-1}}(S) \: \GL_n(S) \to \GL_1(S)=S^{\times}
\]
given by the formula
\begin{equation} \label{eq.identifyingxip-1}
\xi_{\bm{p-1}}(\bm{\mu}) = \xi_{\bm{p-1}}(S)(\bm{\mu}) \coloneqq \sum_{\substack{a \: \{1,\ldots,n\} \times \{1,\ldots,n\} \to \{0,\ldots,p-1\} \\ \sum_{k} a_{kl} = p-1 = \sum_l a_{kl}}  } \frac{(p-1)!^n}{\prod_{k,l} a_{lk}!} \prod_{k,l} \mu_{lk}^{a_{lk}}.
\end{equation}
Moreover, we may think of $\xi_{\bm{p-1}}$ as a morphism between group $\bF_p$-varieties\footnote{That is, of geometrically reduced algebraic groups over $\bF_p$.}
\[
\xi_{\bm{p-1}} \: \GL_n \to \GL_1=\bG_{\mathrm{m}}.
\]

\begin{claim}
\label{clai.Mcompatiblewbasechange} $\xi_{\bm{p-1}}$ is further a group-scheme homomorphism. 
\end{claim}
\begin{proof}[Proof of the claim]
We may assume that $S$ is a separably closed field $\kay$ and show that
\[
\xi_{\bm{p-1}}(\bm \mu \cdot \bm \nu) = \xi_{\bm{p-1}}(\bm \mu) \cdot \xi_{\bm{p-1}}(\bm \nu)
\] 
for all $\bm{\mu}, \bm{\nu} \in \GL_n(\kay)$. See the discussion on density in \cite[pp. 9--13]{MilneAlgebraicGroups}. By gaussian elimination, $\bm \mu$ the product of elementary matrices. Thus, we may assume that $\bm{\mu}$ is an elementary matrix. We proceed with a case by case analysis. To this end, let $\bm{\epsilon}_{ij}$ denote the matrix whose only nonzero entry is the $i,j$ entry which is a $1$. Let us $\bm{\iota}$ denote the identity.

\emph{Assume that $\bm \mu = \bm{\iota} + (c-1) \bm{\epsilon}_{ii}$ with $c \in \kay^{\times}$.} In computing $\xi_{\bm{p-1}}(\bm \mu)$, the only matrices with $a_{kl} = 0$ for $k \neq l$ yield a non-trivial contribution. Thus, we only have the summand where all $a_{kk} = p-1$ and we obtain $\xi_{\bm{p-1}}(\bm \mu) = c^{p-1}$.
The formula for $\xi_{\bm{p-1}}(\bm {\mu} \cdot \bm{\nu})$ is
\begin{align*} \xi_{\bm{p-1}}(\bm {\mu} \cdot \bm{\nu}) &= \sum_{\substack{a \: \{1,\ldots,n\} \times \{1,\ldots,n\} \to \{0,\ldots,p-1\} \\ \sum_{k} a_{kl} = p-1 = \sum_l a_{kl}}  } \frac{(p-1)!^n}{\prod_{k,l} a_{k,l}!} \prod_{l} (c\nu_{il})^{a_{il}}\prod_{k\neq i, l}  \nu_{kl}^{a_{kl}}\\
&= \sum_{\substack{a \: \{1,\ldots,n\} \times \{1,\ldots,n\} \to \{0,\ldots,p-1\} \\ \sum_{k} a_{kl} = q-1 = \sum_l a_{kl}}  } \frac{(p-1)!^n}{\prod_{k,l} a_{k,l}!} c^{p-1} \prod_{k, l}  \nu_{kl}^{a_{kl}} = c^{p-1} \xi_{\bm{p-1}}(\bm {\nu}).\end{align*}

\emph{Assume that $\bm \mu$ is a permutation matrix  $\bm \mu = \bm{\iota} - \bm{\epsilon}_{ii} - \bm{\epsilon}_{jj} + \bm{\epsilon}_{ij} + \bm{\epsilon}_{ji}$.} A similar argument shows that $\xi_{\bm{p-1}}(\bm \mu) = 1$. For the assertion that $\xi_{\bm{p-1}}(\bm {\mu} \cdot \bm{\nu}) = \xi_{\bm{p-1}}(\bm{\nu})$, note that for a matrix $a$ as in the summation, the matrix $a'$ where the $i$th and $j$th rows are switched also occurs in the summation.

\emph{Assume that $\bm \mu = \bm{\iota} + c \bm{\epsilon}_{ij}$ with $c \in \kay$ and $i \neq j$.}
For $\xi_{\bm{p-1}}(\bm \mu)$, we note that only the matrices with $a_{kl} = 0$, except possibly for $k = l$ or $(k,l) = (i,j)$, produce a contribution. If $a$ is such a matrix, then necessarily $a_{ii} = p -1 - a_{ij}$, but $a_{ii}$ is the only entry in the $i$th column, so that $a_{ii} = p-1$ and $a_{ij} = 0$. Thus, there is only the summand for $a = (p-1)\bm{\iota}$, and so $\xi_{\bm{p-1}}(\bm \mu) =1$. Consider the linear change of variables
\[ 
\begin{bmatrix} s_1 \\ \vdots \\ s_n \end{bmatrix} \coloneqq \bm \mu \cdot \bm \nu \begin{bmatrix} t_1 \\ \vdots \\ t_n\end{bmatrix}
\]
in the field of rational functions $\kay(t_1,\ldots,t_n)$
A computation as the one leading to \autoref{Claim.CombinatorialFormula} shows that the coefficient of $\bm t^{\bm {p-1}}$ in $\bm s^{\bm {p-1}}$ is $\xi_{\bm{p-1}}( (\bm{\mu} \cdot \bm{\nu})^\top) = \xi_{\bm{p-1}}(\bm{\mu} \cdot \bm{\nu})$. On the other hand,
\begin{align*}
s_1^{p-1} \cdots s_n^{p-1} &=   \left(\sum_{l=1}^n (\nu_{il} + c \nu_{jl}) t_l\right)^{p-1} \cdot \prod_{k=1, k \neq i}^n\left(\sum_{l=1}^n\nu_{kl} t_l\right)^{p-1}\\&=
 \Bigg( \underbrace{\sum_{l=1}^n \nu_{il}t_l}_u + \underbrace{c \sum_{l=1}^n \nu_{jl} t_l}_v\Bigg)^{p-1} \cdot \prod_{k=1, k \neq i}^n\left( \sum_{l=1}^n \nu_{kl} t_l\right)^{p-1}\\
 &= \left(\sum_{r=0}^{p-1} \binom{p-1}{r} u^{p-1-r} v^r \right) \cdot \prod_{k=1, k \neq i}^n\left( \sum_{l=1}^n \nu_{kl} t_l\right)^{p-1}.
\end{align*}
For any summand in the left hand term with $r \geq 1$, we obtain an expression of the form
\[\bigg(\sum_{l=1}^n \nu_{jl} x_l\bigg)^p \cdot v = \sum_{l=1}^n \nu_{jl}^p x_l^p \cdot v\]
which does not contribute to the coefficient of $t_1^{p-1} \cdots t_n^{p-1}$. Hence, $\xi_{\bm{p-1}}(\bm{\mu} \cdot \bm{\nu}) = \xi_{\bm {p-1}}(\bm \nu)$.
\end{proof}

Thus, we are left with proving that $\xi_{\bm{p-1}} = \det^{p-1}$ as homomorphisms $\GL_n \to \bG_{\mathrm{m}}$ of group $\bF_p$-varieties. For this, it suffices to check equality at $\bar{\bF}_p$-points; see \cite[pp. 9--13]{MilneAlgebraicGroups}. In particular, we may assume that $S$ is an algebraically closed field $\kay$ and show that $\det(\bm \mu)^{p-1} = \xi_{\bm{p-1}}(\bm \mu)$ holds for all $\bm{\mu} \in \GL_n(\kay)$. By \autoref{clai.Mcompatiblewbasechange}, this identity is invariant under similarity/conjugacy of matrices. Hence, we may assume that $\bm \mu$ is in its Jordan normal form. 

Observe that it suffices to consider a single Jordan block with eigenvalue $\lambda$. Indeed, this is clear for $\det(\bm \mu)^{p-1}$. 
For $\xi_{\bm{p-1}}(\bm \mu)$, note that each matrix $a_{kl}$ occurring in \autoref{eq.identifyingxip-1} that yields a non-trivial contribution to $\xi_{\bm{p-1}}(\bm \mu)$ also has the same block diagonal form as $\bm \mu$ (if $a_{kl} \neq 0$ with $l \neq k, k+1$, then $\mu_{kl}^{a_{kl}} = 0)$.) Thus, it suffices to show that, in this case, $\xi_{\bm{p-1}}(\bm \mu) = \lambda^{n(p-1)}$.

In the summation of \autoref{eq.identifyingxip-1}, only matrices with only $a_{kk} \neq 0$ or $a_{k k+1}\neq 0$ yield a non-trivial contribution. If $a_{11} \neq p-1$, then the first column cannot add up to to $p-1$. Hence, $a_{11} = p-1$ and then $a_{12} = 0$. Continuing in this way we see that all $a_{kk}$ must be $p-1$ and all $a_{kk+1}$ have to be zero. Thus $\xi_{\bm{p-1}}(\bm \mu) = \lambda^{n(p-1)}$; as claimed.

This finishes the proof of the theorem.
\end{proof}

\begin{remark}
We note that one can also extract a combinatorial identity from the above theorem.
Letting $S_n$ denote the $n$-th symmetric group, we have:
\begin{align*}
    (\det(\mu_{kl}))^{p-1} &= \left(\sum_{\sigma \in S_n} \sgn \sigma \cdot \mu_{1\sigma(1)} \cdots \mu_{n\sigma(n)}\right)^{p-1}\\
    &= \sum_{\substack{b\: S_n \to \{0,\ldots,p-1\} \\ \sum_{\sigma \in S_n} b_{\sigma} = p-1}} \frac{(p-1)!}{\prod_{\sigma \in S_n} b_{\sigma}!} \prod_{\sigma \in S_n} \big(\sgn \sigma \cdot \mu_{1\sigma(1)} \cdots \mu_{n\sigma(n)}\big)^{b_{\sigma}} \\
    &= \sum_{\substack{b\: S_n \to \{0,\ldots,p-1\} \\ \sum_{\sigma \in S_n} b_{\sigma} = p-1}} \frac{(p-1)!}{\prod_{\sigma \in S_n} b_{\sigma}!}  \prod_{\sigma \in S_n}  \big(\sgn \sigma \big)^{b_{\sigma}} \prod_{k,l} \mu_{kl}^{\sum_{\sigma, \sigma(k) = l} b_{\sigma}}
\end{align*}
Thus, the equality $\xi_{\bm{p-1}}(\bm{\mu}) = \det(\bm \mu)^{p-1}$ is equivalent to the following congruence. Let $(a_{kl})$ be an $n \times n$ matrix with entries in $\{0,\ldots, p-1\}$ whose rows and columns each add up to $p-1$, that is, $\sum_k a_{kl} = p-1 = \sum_l a_{kl}$ (a sort of arithmetic doubly stochastic matrix). Then:
\[
\frac{(p-1)!^n}{\prod_{k,l} a_{k,l}!} \equiv \sum_{\substack{b \: S_n \to \{0,\ldots,p-1\} \\ \sum_{\sigma} b_{\sigma} = p-1 \\  \sum_{\sigma(k)=l} b_{\sigma} = a_{k,l}}} \frac{(p-1)!}{\prod_{\sigma \in S_n} b_{\sigma}!}  \prod_{\sigma \in S_n}  \big(\sgn \sigma \big)^{b_{\sigma}} \qquad \mod p
\]
In other words, let $M_n$ be the set of $n \times n$ matrices with entries in $\{0,\ldots, p-1\}$ whose rows and columns add up to $p-1$ and $B_n$ be the set of functions $b\: S_n \to \{0,\ldots,p-1\}$ such that $\sum_{\sigma \in S_n} b_{\sigma} = p-1$. There is a mapping:
\[
\xi\: B_n \to M_n, \quad b \mapsto \sum_{\sigma \in S_n} b_{\sigma} P_{\sigma}
\]
where $P_{\sigma}$ denotes the matrix obtained by permuting the columns of the identity according to $\sigma$ (i.e., $(P_{\sigma})_{kl} = \delta_{k\sigma(k)}$). Then, for all $a \in M_n$, it follows that
\[
\frac{(p-1)!^n}{\prod_{k,l} a_{k,l}!} \equiv \sum_{b \in \xi^{-1}a} \frac{(p-1)!}{\prod_{\sigma \in S_n} b_{\sigma}!}  \prod_{\sigma \in S_n}  \big(\sgn \sigma \big)^{b_{\sigma}} \qquad \mod p.
\]
This is a combinatorial identity that is independent of $S, R$. If there were another way to show this, we would have another proof for $\xi_{\bm{p-1}} = \det^{p-1}$.
\end{remark}

\subsection{Relative Cartier operators} \label{Sec.RelCartierOperartor}
Let $f\: X \to Y$ be a regular $F$-finite morphism of locally noetherian schemes, as before. Given a coherent $\sO_Y$-module $\sF$, we set \[
f^! \sF \coloneqq \omega_f \otimes f^\ast \sF.
\]
This defines a covariant functor $f^!$ from coherent sheaves on $Y$ to coherent sheaves on $X$. Our next goal is to construct a \emph{relative Cartier operator}, which is a natural transformation 
\[
\varkappa^e \: F_{f \ast}^e f^!  \longrightarrow f^{(q)!}  \coloneqq \omega_{f^{(q)}} \otimes f^{(q)*} =  G^{e\ast}_f \omega_f \otimes f^{(q)*}.
\] 
Note that $f^{(q)}$ is another regular $F$-finite morphism of locally noetherian schemes.
Further, observe that this recovers the classical Cartier operator by plugging in $f\:X \to Y=\Spec \kay$ a smooth morphism to a perfect groundfield $\kay$ and taking $\kappa^e_X = \varkappa^e_{\kay}$. We will also describe it locally in explicit terms in \autoref{le.explicitrelativecartierstructureregular} below.

For the construction of $\varkappa^e$, consider first the Grothendieck trace of $F^e_f$
\begin{equation}
\label{eq.relativeCartier1}
\Tr_{f^{(q)*}\sF} \: F_{f*}^e F_f^{e,!} f^{(q)*}\sF \to f^{(q)*}\sF. 
\end{equation}
Plugging in $f^{(q)*} \sF$ in the natural isomorphism $\gamma_e$, which was defined in \autoref{thm.regulardualizing}, yields
\[
\gamma_{e,f^{(q)*} \sF} \: \omega_f^{1-q} \otimes f^* \sF \xrightarrow{\sim}  F^{e,!}_f f^{(q)*} \sF.
\]
Pushing it forward along $F^e_{f}$ and composing with \autoref{eq.relativeCartier1} then yields a new trace
\begin{equation} \label{eq.relativeCartier2}
F_{f*}^e \bigl( \omega_f^{1-q} \otimes f^* \sF \bigr) \xrightarrow{\sim}   F_{f \ast}^e F^{e,!}_f f^{(q)*} \sF \xrightarrow{\Tr} f^{(q)*}  \sF.
\end{equation}
 Tensoring \autoref{eq.relativeCartier2} by $G^{e\ast}_f \omega_f$ yields
\[
\varkappa^e_{\sF} \: F_{f \ast}^e \big(\omega_f \otimes f^\ast \sF\big) = F_{f \ast}^e f^! \sF \longrightarrow G^{e\ast}_f \omega_f \otimes f^{(q)*}  \sF = f^{(q)!} \sF,
\]
by the projection formula and noting that $\omega_f^q = F^{e\ast}_X \omega_f = F_f^{e \ast} G^{e\ast}_f \omega_f$. In a nutshell, $\varkappa^e$ is the natural transformation $\Tr_{f^{(q)*}}$ tensored by $\omega_{f^{(q)}} = G^{e\ast}_f \omega_f$. 

Conceptually, the main use of $\varkappa^e$ for our purposes is that its pushforward along $G^e_f$ allows us to swap inverse images with Frobenius pushforwards, as the following proposition establishes.

\begin{proposition}
    With notation as above, $G^e_{f \ast} \varkappa^e$ is a natural transformation
\begin{equation}\label{eq.PushforwardRelativeCartier}
G^e_{f\ast} \varkappa^e \: F_{X,*} f^!  \longrightarrow  f^! F_{Y,*}^e .\end{equation}
\end{proposition}
\begin{proof}
Note that $G^e_{f\ast} F^e_{f\ast} = F^e_{X\ast}$. On the other hand, using the projection formula and the natural isomorphism $G^e_{f\ast} f^{(q)*}= f^\ast F_{Y\ast}^e$, we obtain the natural isomorphisms
\[
G^e_{f\ast}f^! = G^e_{f\ast} \big( G^{e\ast}_f \omega_f \otimes f^{(q)*}  \big) = \omega_f \otimes  G^e_{f\ast} f^{(q)*}  = \omega_f \otimes f^\ast F^e_{Y\ast}  = f^! F^e_{Y\ast}.
\]
Thus, up to a natural isomorphism on the source, $G^e_{f\ast} \varkappa^e$ has the desired form.
\end{proof}

We next highlight how to calculate $\varkappa^e$ on local affine charts on which $f$ admits a $d/p$-basis.

\begin{proposition}[Local description of the Cartier operators]
\label{le.explicitrelativecartierstructureregular}
With notation as above, suppose that $f\: \Spec S \to \Spec R$ is affine and admits a $d/p$-basis $x_1, \ldots, x_n \in S$.  Letting $M$ be a finitely generated $R$-module, the Cartier operator
\[ 
\varkappa^e = \varkappa^e_M \: F^e_{f \ast}f^! M \longrightarrow f^{(q)!} M
\] is explicitly given by the $S_{F^e_*R}$-linear map
\begin{align*}
F^e_{f*}(S \cdot \mathrm{d} x_1 \wedge \cdots \wedge \mathrm{d} x_n \otimes M) &\to S_{F^e_* R} \cdot \mathrm{d} x_1 \wedge \cdots \wedge \mathrm{d} x_n \otimes M ,\\
F^e_{f*}(s \cdot \mathrm{d} x_1 \wedge \cdots \wedge \mathrm{d} x_n \otimes m ) &\mapsto \Phi^e_f(F^e_{f*}s) \cdot \mathrm{d} x_1 \wedge \cdots \wedge \mathrm{d} x_n \otimes m .
\end{align*} 
Moreover, the map
\[G^e_{f\ast} \varkappa^e_M\: F_{S\ast}^e f^! M \longrightarrow f^! F^e_{R*}M 
\] is explicitly given by
\begin{align*}
F^e_{S*}(S \cdot \mathrm{d} x_1 \wedge \cdots \wedge \mathrm{d} x_n \otimes M) &\to S \cdot \mathrm{d} x_1 \wedge \cdots \wedge \mathrm{d} x_n \otimes F^e_{R*}M ,\\
F^e_{S*}(s \cdot \mathrm{d} x_1 \wedge \cdots \wedge \mathrm{d} x_n \otimes m ) &\mapsto \sum_i s_i \cdot \mathrm{d} x_1 \wedge \cdots \wedge \mathrm{d} x_n \otimes F^e_{R*}r_i m 
\end{align*} 
where $\Phi^e_f(F^e_{f*}s) = \sum_i s_i \otimes F^e_{R*} r_i.$
\end{proposition}

\autoref{le.explicitrelativecartierstructureregular} shows that $\varkappa^e_{\sF}$ is essentially $\kappa^e_f \otimes_{\sO_Y} \sF$ where \[
\kappa^e_f \coloneqq \varkappa^e_{\sO_Y} \: F^e_{f*} \omega_f \to \omega_{f^{(q)}} \]
is the \emph{Cartier operator of $f$}. Furthermore, it has the following generating property.

\begin{corollary} \label{le.relativeCartierGenerates}
    With notation as above, the map 
    \[
\sO_X \xrightarrow{1 \mapsto \kappa^e_f}    \ssHom_{X^{(q)}}(F^e_{f*} \omega_f,\omega_{f^{(q)}})
    \]
    is an isomorphism of $\sO_X$-modules.
\end{corollary}

\subsection{Pulling back Cartier modules}

Having constructed the relative Cartier operator, we can explain how to pull back Cartier linear maps. If $\phi\: F_{Y\ast}^e \sF \to \sF$ is an $\sO_Y$-linear map, then we define a Cartier linear map $\phi^{!}\: F_{X\ast}^e f^! \sF \to f^! \sF$ by setting
\[ 
\phi^! \coloneqq f^! \phi \circ G_{f\ast}^e \varkappa_{\sF}^e.
\]
In particular, in the setup of \autoref{le.explicitrelativecartierstructureregular}, the map $\phi^{!}$ is explictly given locally by
\begin{equation} \label{eqn.DefOfPhiUpperShriek}
    \phi^! \: F^e_{S*}(s \cdot \mathrm{d} x_1 \wedge \cdots \wedge \mathrm{d} x_n \otimes m ) \mapsto \sum_i s_i \cdot \mathrm{d} x_1 \wedge \cdots \wedge \mathrm{d} x_n \otimes \phi(F^e_{R*}r_i m),
\end{equation}
where $\Phi^e_f(F^e_{f*}s) = \sum_i s_i \otimes F^e_{R*} r_i$.

Let us recall that a \emph{Cartier algebra} on $Y$ is a sheaf of $\mathbb{N}$-graded rings $\sC = \bigoplus_{e\in \bN} \sC_e$ with $\sC_{0} = \sO_Y$ that is a quasi-coherent $\sO_Y$-bimodule with $r \cdot \varphi = \varphi \cdot r^{q}$ for local sections $\varphi \in \sC_{e}$. In particular, the right $\sO_Y$-module structure determines the left $\sO_Y$-module structure. Thus, whenever we view $\sC$ as an $\sO_Y$-module, we refer to its right $\sO_Y$-module structure. For example, $f^*\sC$ is the Cartier algebra on $X$ obtained by pulling back $\sC$ as a right $\sO_Y$-module. See \cite[Proposition 5.3]{BlickleStablerFunctorialTestModules} for more details. Given a quasi-coherent sheaf $\sF$ on $Y$, its \emph{full Cartier algebra} is defined as
\[
\sC_{\sF} \coloneqq \sO_{Y} \oplus \bigoplus_{e \in \bN_+} \ssHom_Y(F^e_{Y*} \sF,\sF).
\]

For $\sF$ to become a \emph{Cartier $\sC$-module}, all that is needed is a graded homomorphism $\sC \to \sC_{\sF}$ plus the condition that $\sF$ is coherent. This gives rise to the category $\sC$-mod. 

The assignment $\phi \mapsto \phi^!$ defined above gives rise to an $\sO_Y$-linear homomorphism
\[
\sC_{\sF,e} \xrightarrow{\phi \mapsto \phi^!} f_* \sC_{f^!\sF,e}
\]
and, moreover, to a graded homomorphism $\sC_{\sF} \to f_* \sC_{f^!\sF}$. In particular, if $\sF$ is a $\sC$-mod with structural homomorphism $\sC \to \sC_{\sF}$, the composition
\[
\sC \to \sC_{\sF} \to f_* \sC_{f^! \sF}
\]
is adjoint to a map
\[
f^* \sC \to \sC_{f^!\sF},
\]
and so $f^!\sF$ becomes a Cartier $f^*\sC$-module. To sum up, the above defines the functor
\[
f^!\: \sC\text{-mod} \to f^*\sC\text{-mod},
\]
which lets us ``pull back'' Cartier modules along regular maps. 

Let us recall that an analogous formalism can be constructed for finite covers. See \cite[\S2.2]{CarvajalStablerFsignaturefinitemorphisms}. In a nutshell, if $g\: \Spec A \to \Spec R$ is a finite morphism and $\phi\: F^e_* M \to M$ is a Cartier action, then we define a map
\[
\phi^! \: F_*^e g^! M \xrightarrow{F^e_* \mu \mapsto \phi \circ F^e_* \mu \circ F^e} g^! M
\]
for all $\mu \in g^! M \coloneqq \Hom_R(S,M)$. In the next proposition, we provide certain compatibility between these two constructions. This will be essential in the next section.

\begin{proposition}[{\cf \cite[Lemma 4.5]{stablerunitftestmodules}}]
\label{le.cartericompatibility}
Let $f\: X \to Y$ be  a regular $F$-finite morphism of locally noetherian schemes and $g\: Z \to Y$ a finite morphism. Consider a cartesian diagram
\[
\xymatrix{W\ar[r]^{g'} \ar[d]_{f'} &X \ar[d]^f\\ Z \ar[r]^g& Y}
\]
The canonical natural isomorphism $f'^! g^! \to g'^! f^!$ extends to the category of Cartier modules.
\end{proposition}
\begin{proof}
We start by observing that $f'$ satisfies the same hypothesis as $f$ (\autoref{co.ffiniteregularbasechange}) and likewise for $g'$ and $g$. Noting that $\omega_{f'} = g'^* \omega_f$ and $f^*g_* \sO_Z = g'_* f'^*\sO_Z = g'_* \sO_W$, the natural isomorphism is given by the canonical maps of $\sO_W$-modules
\begin{align*}
\ssHom_Y(g_* \sO_Z, \sF)\otimes_{\sO_Y} \omega_f = f^*\ssHom_Y(g_* \sO_Z, \sF)\otimes_{\sO_X} \omega_f 
&\xrightarrow{\sim} \ssHom_Y(g'_* \sO_W, f^*\sF)\otimes_{\sO_X} \omega_f \\
&\xrightarrow{\sim} \ssHom_X(g'_* \sO_W, \omega_f \otimes f^* \sF),
\end{align*}
where the first arrow is an isomorphism as $f$ is flat and $g$ is finite, while the second one is an isomorphism as $\omega_f$ is invertible. More succinctly, if we have affine charts $f \: \Spec S \to \Spec R$ and $g \: \Spec A \to \Spec R$ with $B \coloneqq A \otimes_R S$, the natural isomorphism of $B$-modules is
\[
\Hom_R(A,M) \otimes_R \omega_f \xrightarrow{\mu \otimes \delta \mapsto [a \otimes s \mapsto \mu(a) \otimes s\delta ]} \Hom_S(B,M \otimes_R \omega_f).
\]
In order to show the compatibility of this natural isomorphism with Cartier structures, let us fix an $R$-linear map $\phi \: F^e_* M \to M$. We may work locally and assume that $f$ admits a $d/p$-basis $x_1,\ldots,x_n\in S$, and we may write $\omega_f = S$ for notation ease. Then, by \autoref{eqn.DefOfPhiUpperShriek}, the induced Cartier action on $\Hom_R(A,M) \otimes_R S$
is given by
\[
\phi^! \: F^e_*(\mu \otimes s) \mapsto \sum_i (\phi \circ F^e_* r_i\mu \circ F^e) \otimes s_i
\]
where $\Phi^e_f(F^e_{f*}s) = \sum_i s_i \otimes F^e_{R*} r_i$. Similarly, the Cartier action on $
\Hom_S(B,M \otimes_R S)$ is
\[
\phi^!\: \nu \mapsto \Biggl[b \mapsto \sum_j\sum_i \phi(F^e_* r_{ji} m_j) \otimes s_{ji} \Biggr]
\]
where $\nu(b^q) = \sum_j m_j \otimes s_i$ and  $\Phi^e_f(F^e_{f*}s_j) = \sum_i s_{ji} \otimes F^e_{R*} r_{ji}$.

We need to show that the following diagram is commutative
\[
\xymatrix{
F^e_* \Hom_R(A,M) \otimes_R S \ar[r]^-{\phi^!} \ar[d]^-{\sim} & \Hom_R(A,M) \otimes_R S \ar[d]^-{\sim}\\
F^e_* \Hom_S(B,M \otimes_R S) \ar[r]^-{\phi^!} & \Hom_S(B,M \otimes_R S)
}
\]
Let $F^e_*(\mu \otimes s)$ be an element in the upper-left corner with $\Phi^e_f(F^e_{f*}s) = \sum_i s_i \otimes F^e_{R*} r_i$. Following either path, the above tells us that its image is the map
\[
a \otimes s' \mapsto \sum_i \phi(F^e_* r_i \mu(a^q)) \otimes s_i s'.
\]
The claimed compatibility then holds.
\end{proof}
Finally, we verify that our construction is compatible with the factorization in \autoref{co.regularpbasisfactorization}.

\begin{proposition}
\label{le.regularsmoothpristinefactorizationcompatible}
Let $f = \Spec \theta\: \Spec S \to \Spec R$ be a regular homomorphism of noetherian rings that admits a $p$-basis $x_1,\ldots, x_n \in S$. Let 
\[
\Spec S \xrightarrow{h} \mathbb{A}^n_R \xrightarrow{g} \Spec R
\] 
with $h$ pristine be a factorization of $f$ as in \autoref{co.regularpbasisfactorization}. Let $\sC$ be a Cartier algebra on $R$. If $M$ is a $\sC$-module, then $f^! M$ and $h^! g^! M = h^* g^! M$ are naturally isomorphic as $f^\ast \sC$-modules. 
\end{proposition}
\begin{proof}
We have a natural isomorphism $f^\ast \sC \cong h^\ast g^\ast \sC$ that we use to identify them. By definition, $f^! M = \omega_f \otimes_S f^\ast M$ and $h^! g^! M =\omega_h  \otimes_S h^\ast g^\ast M$. Write $t_1, \ldots, t_n$ for a set of variables of $g$. Since $\omega_g = R[t_1,\ldots,t_n]$ we may also naturally identify the $S$-modules $f^!M$ and $h^!g^! M$. It only remains to show that the Cartier action is compatible with this identification. This follows from \autoref{le.explicitrelativecartierstructureregular} and \cite[Theorem 5.5 (c)]{BlickleStablerFunctorialTestModules}.
\end{proof}

\begin{proposition}
\label{pro.shriekfullhomset}
Let $f \: \Spec S \to \Spec R$ be a regular homomorphism of noetherian $F$-finite rings such that $\omega_f$ is free. There is a non-canonical isomorphism $f^\ast \sC_M \cong \sC_{f^!M}$ for a finitely generated $R$-module $M$.
\end{proposition}
\begin{proof}
We have a non-canonical natural isomorphism $f^! \cong f^*$, which depends only on the choice of an isomorphism $\omega_f \cong S$. The assertion boils down to:
\begin{align*} 
\Hom_S(F_\ast^e f^! M, f^!M) \cong \Hom_S(f^!M, F^{e,!} f^! M) &= \Hom_S(f^!M, F^{e,!}_f G^{e,!}_f f^!M)\\
& \cong  \Hom_S(f^!M, F^{e,!}_f f^{(q)!} F^{e,!} M) \\
& \cong  \Hom_S(f^!M, f^! F^{e,!} M)\\
 &\cong f^*  \Hom_R(M, F^{e,!} M)\\
 &\cong f^* \Hom_R(F^e_*M,  M).
\end{align*}
\end{proof}

\section{Commutativity of Regular Pullbacks with Functorial Test Modules}
\label{Testmodulesregularpullbacks}

Before turning to the proof of our main result, we first review several concepts from the theory of Cartier modules. Since all of these constructions are compatible with localization, it suffices to formulate them in the context of an affine scheme.

\begin{definition}
\label{def.testmodulesummary}
Let $(M,\sC)$ be a Cartier module over an $F$-finite ring $R$.
\begin{enumerate}[(a)]
\item By \cite[Proposition 2.13]{BlickleTestIdealsViaAlgebras}, the descending chain \[M \supset \sC_+ M \supset (\sC_+)^2 M \supset \cdots\] stabilizes. We denote its stable member by $\underline{M}$ or $\sC_+^e M$ for $e \gg 0$. This is also called the \emph{non-$F$-pure submodule} and is denoted by $\upsigma(M)=\upsigma(M, \sC)$ in the literature. The Cartier module $M$ is called \emph{nilpotent} if $\underline{M} = 0$. A homomorphism $\varphi\: M \to N$ of $\sC$-modules is called a \emph{nil-isomorphism} if both $\ker \varphi$ and $\coker \varphi$ are nilpotent.
\item We say that $M$ is \emph{$F$-pure} if $M = \underline{M}$ (equivalently $\sC_+ M = M$). In that case, $\Ann_R M$ is a radical ideal (\cite[Lemma 2.19]{BlickleTestIdealsViaAlgebras}) and one may replace $R$ with $R/\Ann_R M$ and assume that $\Supp M = \Spec R$.
\item{Note that $H^0_\p(M)$ is a $\sC$-module. We write $\Assoc M$ for the set of associated primes of the underlying $R$-module. The set of associated primes of the Cartier module $M$ is denoted by $\Ass M$, which consists of those $\p \in \Assoc M$ for which $H^0_\p(M)_\p$ is not nilpotent.}
\item The \emph{test module} $\uptau(M, \sC)$ of $(M, \sC)$ is defined as the smallest $\sC$-submondule $N \subset M$ such that $H^0_\p(N)_\p \subset H^0_\p(M)_\p$ is a nil-isomorphism for all $\p \in \Ass(M)$. We say that $M$ is \emph{$F$-regular} if $\uptau(M, \sC) = M$.
\item Letting $\Ass(M)=\{\p_1, \ldots, \p_n\}$, elements $c_1, \ldots, c_n \in R$ are called a \emph{sequence of test elements} if $c_i \notin \p_i$ and $\upsigma(H^0_{\p_i}(M))_{c_i}$ is $F$-regular for all $i$. If $\uptau(M, \sC)$ exists, then
\[ \uptau(M, \sC) = \sum_{i=1}^n \sum_{e \geq e_0} \sC_e c_i \underline{H^0_{\p_i}(M)}\] for any $e_0 \geq 0$. See \cite[Theorem 3.1]{BlickleStablerFunctorialTestModules}.
\end{enumerate}
\end{definition}

\begin{remark}
    In computing the test module, one may replace $M$ by $\underline{M}$ and thus assume that $\Supp M = \Spec R$ (by replacing $R$ with $R/\Ann M$). Then, if all associated primes of $M$ are minimal, the computation of $\uptau(M)$ simplifies (and we also get an existence result): $\uptau(M, \sC)$ exists if and only if there is $c$ avoiding every minimal prime of $M$ such that $M_c$ is $F$-regular. In this case, one has $\uptau(M, \sC) = \sC c^t M$ for all $t \geq 1$ (see \cite[Theorem 3.11]{BlickleTestIdealsViaAlgebras}.  We refer the reader to \cite{BlickleStablerFunctorialTestModules} for a more elaborate discussion concerning test modules.
\end{remark}

\begin{lemma}
\label{le.cartieractionshriek}
Let $f\: X \to Y$ be a regular morphism of $F$-finite locally noetherian schemes. If $\sF$ is a $\sC_Y$-module, then $\sC_{X+} f^! \sF = f^! \sC_{Y+} \sF$, where $\sC_X = f^\ast \sC_Y$. Put differently, $\upsigma(f^! \sF, \sC_X) = f^! \upsigma(\sF, \sC_Y)$.
\end{lemma}
\begin{proof}
We may assume that both $X,Y$ are affine and that $\Omega_f$ admits a differential basis. We may then use \autoref{pro.fsmoothdbasisfactorization}, \cite[Lemma A.1]{CarvajalRojasStablerPristineMorphisms}, and \cite[Lemma 6.1]{BlickleStablerFunctorialTestModules}. 
\end{proof}

\begin{lemma}
\label{le.AssPrimesRegularMorphism}
Let $f\: \Spec S \to \Spec R$ be a regular morphism of $F$-finite noetherian schemes. Let $M$ be a $\sC$-module over a Cartier $R$-algebra $\sC$. Then, $\Ass f^! M$ consists of those primes $\q$ that are minimal over $\p S$ for some $\p \in \Ass M$.
\end{lemma}
\begin{proof}
We may assume $\Omega_f$ to be free. By \cite[\href{https://stacks.math.columbia.edu/tag/0312}{Tag 0312}]{stacks-project}, we have 
\[
\Assoc_S(M \otimes_R \omega_f) = \bigcup_{\p \in \Assoc(M)} \Assoc_S(\omega_f/\p \, \omega_f).
\]
Note that $\omega_f/\p \omega_f \cong S/ \p S$ and the associated primes of this $S$-module are the minimal primes of $\p S$ (by \autoref{co.ffiniteregularbasechange}, $R/\p \to S/\p S$ is regular and thus reduced, so the set of zero divisors of $S/\p S$ is the union of its minimal primes). It remains to show that if $\p \in \Assoc M$ and $\q$ is a minimal prime of $\p S$, then $H^0_\q(f^!M)_\q$  nilpotent if and only if so is $H^0_\p(M)_\p $.

Using \cite[Lemma 3.2]{BlickleStablerFunctorialTestModules}, we do the identification 
\[
\underline{H^0_\p(M)} = \underline{i^! M},
\]
where $i\: \Spec R/\p \to \Spec R$ is the closed immersion. Since localization is essentially \'etale and preserves $F$-purity, we may further identify 
\[
\underline{H^0_\p(M)_\p}=j^! \underline{i^! M},
\] where $j\: \Spec (R/\p)_\p \to \Spec R/\p$ is the localization morphism. Likewise for $\underline{H^0_\q(f^!M)_\q}$.

Consider the following commutative diagram
\[
\xymatrix{
\Spec \kay(\p) \ar[r]^j & \Spec R/\p \ar[r]^i& \Spec R\\
f^{-1}(\p) \ar[r] \ar[u]^{f''} & \Spec S/\p S \ar[u]^{f'} \ar[r] & \Spec S \ar[u]_f
}
\]
of cartesian squares, so the vertical arrows are regular. Observe that the canonical morphism  
\[
g\: \Spec (S/\q S)_\q \to f^{-1}(\p)
\] is essentially 
\'etale. Indeed, $\q$ is a minimal prime in $f^{-1}(\p)$ which, since $f''$ is regular and $\kay(\p)$ is a field, corresponds to a connected component. Thus, $g$ is a composition of a localization and an open immersion. By \autoref{le.cartericompatibility}, we have $H^0_\q(f^!M)_\q = g^! f''^!H^0_\p(M)_\p$. The composition $f'' \circ g$ is regular (in particular, flat) and surjective. Thus, \autoref{le.cartieractionshriek} implies that $H^0_\q(f^!M)_\q$ is not nilpotent if and only if $H^0_\p(M)_\p$ is not nilpotent.
\end{proof}

\begin{theorem}
\label{theo.Fregularregularpullback}
Let $f\: X \to Y$ be a regular morphism of locally noetherian $F$-finite schemes and  $\sF$ be a Cartier $\sC$-module on $Y$. Assume either that
\begin{enumerate}
\item $\Ass \sF$ consists of minimal primes of $\Supp \sC^e_+ \sF$ ($e \gg 0)$, or
\item  $\uptau(H^0_\p(\sF))$ exists for all $\p \in \Ass \sF$.
\end{enumerate}
If $(\sF,\sC)$ is $F$-regular, then $(f^! \sF,f^*\sC)$ is $F$-regular. The converse holds if $f$ is surjective.
\end{theorem}
\begin{proof}
We may assume that $f\: \Spec S \to \Spec R$ is affine and that there is a $d/p$-basis. If $f$ is surjective and $f^! M$ is $F$-regular, then $M$ is $F$-pure by \autoref{le.cartieractionshriek}. Hence, we may assume that $M$ is $F$-pure. Moreover, we may base-change to $R/\Ann_R(M)$ using \autoref{co.ffiniteregularbasechange} and so assume that $R$ is reduced and $\Supp M = \Spec R$.

We will use throughout the proof that for a flat morphism $g$ the operations $g^\ast$ and $H^0_I$ commute; in particular, $H^0_I$ commutes with localization. Furthermore, since $\Omega_f$ is locally free, $H^0_I(N \otimes \omega_f) \cong H^0_I(N) \otimes \omega_f$ for every $S$-module $N$. Putting this together, we have $g^! H^0_I(M) = H^0_{IB}(g^! M)$ for $g\: \Spec B \to \Spec A$ any regular $F$-finite map.

If $f$ is surjective, then we may argue just as in \cite[Theorem 6.5]{BlickleStablerFunctorialTestModules} to conclude that $f^! M$ is $F$-regular only if $M$ is $F$-regular. Since $M$ is $F$-pure, $\sC_+^e M = M$, and by our reduction, $\Supp M = \Spec R$. Thus, assuming (a), the result follows from the following claim.
\begin{claim}
\label{claim.mainthm}
If $M$ is an $F$-regular Cartier module whose associated primes are minimal, then $f^!M$ is $F$-regular.
\end{claim}
\begin{proof}[Proof of the claim] By \autoref{co.regularpbasisfactorization}, we may factor $f = g \circ h$ where $g$ is smooth, and $h$ is pristine. Apply now \cite[Lemma 4.6]{stablerunitftestmodules}, \cite[Theorem A.3]{CarvajalRojasStablerPristineMorphisms}, and \autoref{le.regularsmoothpristinefactorizationcompatible}.
\end{proof}

We now proceed with the proof assuming (b).
\begin{claim}
\label{cla.existenceoftestelements}
For every associated prime $\p$ of $M$, there is $c \in R \setminus \p$ such that $\underline{f^! H^0_\p(M)}_c$ is $F$-regular and has only minimal associated primes.
\end{claim}
\begin{proof}[Proof of the claim]
We have the following cartesian diagram 
\[
\xymatrix{\Spec S \ar[r]^f& \Spec R\\ \Spec S/\p S \ar[u]^{i'} \ar[r]^{f'} & \Spec R/\p \ar[u]^i} 
\] 
where $f'$ is regular by \autoref{co.ffiniteregularbasechange}. Note that $\Ass H^0_\p(M)$ consists of associated primes containing $\p$. Hence, we find $c \in R \setminus \p$ such that $\Ass H^0_\p(M)_c$ is just $\p$. Passing to the stable image $\sC_+^e H^0_\p(M)_c$ and possibly choosing a different $c$, we have that $\underline{H^0_\p(M)}_c$ is $F$-regular by \cite[Theorem 3.11]{BlickleTestIdealsViaAlgebras}; where we use the existence of $\uptau(H^0_\p(M))$. 

Identifying $\underline{H^0_\p(\phantom{M})}$ with $i_\ast \underline{i^! \phantom{M}}$ (see \cite[Lemma 3.2]{BlickleStablerFunctorialTestModules}) and using that $\uptau$ commutes with pushforwards along closed immersions, we have reduced the claim to the following situation: if $f\: \Spec S \to \Spec R$ is regular and $N \coloneqq i_\ast \underline{i^! M}_c$ is an $F$-pure Cartier module with only minimal associated primes for which $\uptau(N, \sC)$ exists, then there is $c' \in R$ such that $f^!N_{c'}$ is $F$-regular. This is a consequence of \autoref{claim.mainthm}.
\end{proof}
By \cite[\href{https://stacks.math.columbia.edu/tag/056J}{Lemma 056J}]{stacks-project}, we have
\[
\Supp \underline{f^!H^0_\p(M)} = \Supp f^!\underline{H^0_\p(M)} = \Spec S/\p S.
\]
Since $R$ is excellent, we find $c$ as in \autoref{cla.existenceoftestelements} for which $(R/\p R)_c$ is also normal. Then $(S/\p S)_c$ is also normal and thus of the form $(S/\p S)_c = (S/\q_1)_c \times \cdots \times (S/\q_r)_c$, where the $\q_i$ are the minimal primes of $\p S$. Write $j_a$ for the open (and closed) immersion $\Spec (S/\q_a S)_c \to \Spec (S/\p S)_c$. Since $j_a$ is \'etale we obtain via \autoref{claim.mainthm} that 
\[
j_a^! \underline{f^! H^0_\p(M)}_c = \underline{H^0_{\q_a}(f^!M)_c}
\] is $F$-regular. In particular, writing $c_\p$ for the specific choice of $c$ for $\p$ as above, the $c_\p$ with appropriate repetitions form a sequence of test elements for $f^!M$.

We also note that for $\p \in \Ass M$
\begin{equation}\label{eq.tauregulartransf1}\sum_{\q \in \min \p S} \underline{H^0_\q(f^!M)}_{c_\p} = \underline{f^! H^0_\p(M)}_{c_\p}\end{equation} where we write $\min \p S$ for the set of minimal primes of $\p S$. Since $\p S \subset \q$ for any $\q \in \min \p S$, we also have $H^0_\q(f^!M) \subset H^0_{\p S}(f^!M)$, which implies
\begin{equation} \label{eq.tauregulartransf2} \sum_{\q \in \min \p S} \sC_+ c_\p \underline{H^0_\q(f^!M)} \subset \underline{H^0_{\p S}(f^!M)}.\end{equation}
Note that \autoref{eq.tauregulartransf1} shows that localizing \autoref{eq.tauregulartransf2} at $c_\p$ yields an equality. Thus, by \cite[Lemma 2.2]{BlickleStablerFunctorialTestModules} we get an inclusion
\begin{equation}\label{eq.tauregulartransf3}c_\p \underline{H^0_{\p S}(f^!M)} \subset \sum_{\q \in \min \p S} \sC_+ c_\p \underline{H^0_\q(f^!M)}. \end{equation}
Since the elements $c_\p$ form a sequence of test elements for $f^!M$, we obtain from \cite[Theorem 3.4]{BlickleStablerFunctorialTestModules} and \autoref{eq.tauregulartransf3} by summing over all $\p \in \Ass M$ the inclusion
\begin{equation}\label{eq.tauregulartransf4}\sum_{\p \in \Ass M} c_{\p} \underline{H^0_{\p S}(f^!M)} \subset \sum_{\q \in \Ass f^!M} \sC_+ c_{f(\q)} \underline{H^0_\q(f^!M)} = \uptau(f^!M, f^\ast \sC), \end{equation} where we use \autoref{le.AssPrimesRegularMorphism} to deduce that $\{\q \mid \q \in \min \p S$ for some $\p \in \Ass M\} = \Ass f^!M$.
Since $\uptau(f^!M, f^\ast \sC)$ is a Cartier module, \autoref{eq.tauregulartransf4} then also yields the containment
\begin{equation}\label{eq.tauregulartransf5}\sum_{\p \in \Ass M} f^\ast \sC_+ c_{\p} \underline{H^0_{\p S}(f^!M)} = f^!\Biggl(\sum_{\p \in \Ass M}  \sC_+ c_{\p} \underline{H^0_{\p}(M)}\Biggr) = f^! \uptau(M, \sC) \subset \uptau(f^!M, f^\ast \sC) \end{equation}
where the first equality is due to \autoref{le.cartieractionshriek} and the second is another application of \cite[Theorem 3.4]{BlickleStablerFunctorialTestModules}—for this, recall that the $c_\p$ were chosen in such a way that they form a sequence of test elements for $M$. Since $M$ is assumed $F$-regular, we thus obtain $f^!M \subset \uptau(f^!M, f^\ast \sC)$ from \autoref{eq.tauregulartransf5} and hence equality holds.
\end{proof}

\begin{corollary}
\label{cor.tauregulartransforms}
Let $f\: X \to Y$ be a regular morphism of locally noetherian $F$-finite schemes, and $\sF$ be a Cartier $\sC$-module on $Y$.Then, $f^! \uptau(\sF, \sC) = \uptau(f^! \sF, f^\ast \sC)$ if either
\begin{enumerate}[(a)]
\item$\Ass \sF$ consists of minimal primes of $\Supp \sC^e_+ \sF$ ($e \gg 0)$, or 
\item$\uptau(H^0_\p(\sF), \sC)$ exists for all $\p \in \Ass \sF$.
\end{enumerate}
\end{corollary}
\begin{proof}
Since $f$ is flat, $f^!$ is exact, so we have an inclusion $f^! \uptau(\sF, \sC) \subset f^! \sF$. We want to show first that $\uptau(f^! \sF, f^\ast \sC) \subset f^! \uptau(\sF, \sC)$. To this end, note that, by definition of $\uptau$, we have for any associated prime $\p$ of $\sF$ a nil-isomorphism
\[ H^0_\p\big(\uptau(\sF, \sC)\big)_\p \subset H^0_\p(\sF)_\p.\]
From this, we obtain nil-isomorphisms
\[H^0_\q\big(f^!\uptau(\sF, \sC)\big)_\q \subset H^0_\q(f^!\sF)_\q\] for all associated primes $\q$ of $f^!\sF$ (see proof of \autoref{le.AssPrimesRegularMorphism}). By definition, $\uptau(f^!\sF, f^\ast \sC)$ is the smallest submodule of $f^!\sF$ for which the inclusions
\[H^0_\q\big(\uptau(f^!\sF, f^\ast \sC)\big)_\q \subset H^0_\q(f^!\sF)_\q \] are nil-isomorphisms, so that $\uptau(f^!\sF, f^\ast \sC) \subset f^! \uptau(\sF, \sC)$. Since $f^! \uptau(\sF, \sC)$ is $F$-regular by \autoref{theo.Fregularregularpullback}, the equality holds.
\end{proof}

\begin{remark}
\label{rem.similarresults}
\autoref{theo.Fregularregularpullback} seems to contradict \cite[Theorem 4.11]{BydlonCounterexamplesBertiniTestideals}. However, the example considered there is \emph{not} smooth.
There are a number of related results in the literature:
\begin{enumerate}[(a)]
\item Hashimoto proves that $F$-regularity ascends under a regular morphism in certain circumstances; see \cite[Lemma 3.28]{HashimotoFpurehomomorphisms}. In the $F$-finite setting, our result is stronger, but Hashimoto's result does not require $F$-finiteness.

\item V\'elez proved (\cite{VelezOpennessOfTheFRationalLocus}) that for $R \to S$ a regular morphism of locally excellent rings, if $R$ is $F$-rational, then so is $S$. In the $F$-finite setting, \autoref{cor.tauregulartransforms} generalizes this to an equality of test modules.

\item More generally, Enescu proved (\cite[Theorem 2.24]{EnescuOnTheBehaviorOfFrationalRingsUnderFlatBaseChange})  (under mild additional hypotheses) the following: If $\theta\:R \to S$ is flat and both $R$ and the fibers of $\theta$ are geometrically $F$-rational, then $S$ is $F$-rational.
\end{enumerate}
In light of Enescu's result, it is natural to ask whether our hypothesis that $f$ is regular can similarly be weakened. This is indeed the case if we restrict to the finite type case (see \cite{CarvajalRojasStablerAdjointTestModules}) We note, however, that by \cite{SinghFRegularityDoesNotDeform}, the analog of Enescu's result is not true for $F$-regularity.
\end{remark}

\section{Applications to Constancy Regions of Mixed Test Modules}
\label{sec.constancyregions}

P\'erez proved that for $\kay = \bF_q$ a finite field, the set of $F$-jumping numbers for mixed test ideals over a \emph{regular} finite type $\kay$-algebra is discrete. He also showed that the constancy regions are $p$-fractal functions (see \autoref{def.pfractal} below). See \cite{perezconstancyregions}. The goal of this section is to generalize these results in two ways. First, we drop the regularity hypothesis and treat arbitrary Cartier modules. Second, we use \autoref{cor.tauregulartransforms} to relax the assumption that $\kay$ is a finite field. Given ideals $\mathfrak{a}_1, \ldots , \mathfrak{a}_n $ and $\bm{t} = (t_1, \ldots, t_n) \in \mathbb{R}_{\geq 0}^n$, we write 
\[
\uptau(M, \mathfrak{a}^{\bm{t}}) \coloneqq \uptau(M, \mathfrak{a}_1^{t_1}, \ldots, \mathfrak{a}_n^{t_n}).
\]

Before we can proceed, we need the notion of a \emph{gauge bounded} Cartier algebra. We refer the reader to \cite[Section 4]{BlickleTestIdealsViaAlgebras} for a detailed discussion. For a finite type $\kay$-algebra $R$, a \emph{gauge} is a function $R \to \mathbb{N}$ that one should think of as a replacement for the degree of a polynomial ring. For any finitely generated $R$-module $M$, it induces (non-canonically) a function $M \to \mathbb{N}$ which we also refer to as a gauge.  The result we need from this theory is that if $(M, \sC)$ is \emph{gauge bounded}, then there is a finite-dimensional $\kay$-subspace $V$ of $M$ such that any $F$-pure submodule $N\subset M$ has $R$-module generators in $V$ (\cite[Corollary 4.10]{BlickleTestIdealsViaAlgebras}). Moreover, if $(M, \sC)$ is gauge bounded, then so is $(M, \sC^{\mathfrak{a}^t})$ (\cite[Proposition 4.15]{BlickleTestIdealsViaAlgebras}). With the above assumptions on $R$, every finitely generated Cartier algebra $\sC$ is gauge bounded (\cite[Proposition 4.9]{BlickleTestIdealsViaAlgebras}).
We say that an $R$-module $M$ is generated in gauge $\leq d$ if there are module generators $m_1, \ldots, m_r$ each with gauge $\leq d$.

\begin{lemma}
\label{le.gbfinitelymanytestmodules}
\label{le.discretenessfinitefield}
Let $R$ be of finite type over $\mathbb{F}_q$, $\sC$ be a Cartier algebra, and $M$ be a gauge bounded Cartier module. Given ideals $\mathfrak{a}_1, \ldots, \mathfrak{a}_n \subset R$, the set $\left\{\uptau(M, \sC, \mathfrak{a}^{\bm{t}}) \mid \bm{t} \in [0,T]^n\right\}$ is finite for all $T \geq 0$.
\end{lemma}
\begin{proof}
By gauge boundedness (\cite[Corollary 4.10, Proposition 4.15]{BlickleTestIdealsViaAlgebras}), each of the modules $\uptau(M, \sC, \mathfrak{a}^{\bm{t}})$ is determined by its intersection with $M_{K/(p-1) + 1 + Td}$, where $d$ is such that each $\mathfrak{a}_i$ is generated in gauge $\leq d$ and $K$ is the gauge bound for $\sC$. Since $M_{K/(p-1) + 1 + Td}$ is an $\mathbb{F}_q$-vector space and thus a \emph{finite} set there are only finitely many possibilities for the $\uptau$.
\end{proof}

\begin{remark}
In \autoref{le.gbfinitelymanytestmodules}, we can relax the assumptions on $R$ to be essentially of finite type provided that $(M, \sC)$ is obtained via a localization $R' \to R$ from a gauge bounded pair $(M', \sC')$. This is true if $\sC$ is generated by one element.
\end{remark}

\begin{definition}[\cf {\cite[D\'efinition 4.8.4]{EGA_IV_II}}]
Let $R$ be of finite type over a field $\kay$. We say that a subfield $\el \subset \kay$ is a \emph{field of definition} for $R$ if there is a finite type $\el$-algebra $R'$ such that $R' \otimes_\el \kay = R$. Similarly, given an ideal $\mathfrak{a} \subset R$ (or an $R$-module $M$) we say that $\el$ is a field of definition for $\mathfrak{a}$ (for $M$) if there is an ideal $\mathfrak{a}' \subset R'$ (an $R'$-module $M'$) such that $\mathfrak{a}' \otimes_\el \kay = \mathfrak{a}'$ ($M' \otimes_\el \kay = M$).
\end{definition}

For the rest of this section, we will work in the following setup.
\begin{setup}
\label{set.mixedtestmodules}
Let $\kay$ be an $F$-finite field and $R$ be a finite type $\kay$-algebra with ideals $\mathfrak{a}_1, \ldots, \mathfrak{a}_n \subset R$, and a Cartier module $(M,\sC)$, with $\sC$ finitely generated, that admit a field of definition $\el$ such that: $\kay/\el$ is separable and $\el$ is is a finite separable extension of $\F_q(t_1, \ldots, t_n)$ for some $n$. This includes the case $\F_q= \el$ since $\F_q \to \kay$ is separable as $\F_q$ is perfect; see \cite[\href{https://stacks.math.columbia.edu/tag/0322}{Proposition 0322}]{stacks-project}.
\end{setup}

\begin{theorem}
\label{theo.finitenessconstacyregions}
Working in \autoref{set.mixedtestmodules}, $\left\{ \uptau(M, \sC, \mathfrak{a}^{\bm{t}}) \mid \bm{t} \in [0,T]^n \right\} $ is finite for all $T \geq 0$.
\end{theorem}
\begin{proof}
By our assumption, there is a finite type $\el$-algebra $R'$ such that $R' \otimes_\el \kay \cong R$, a finite $R'$-module $M'$ with $M' \otimes_\el \kay \cong M$, and ideals $\mathfrak{a}_i'$ with $\mathfrak{a}_i' \otimes_\el \kay \cong \mathfrak{a}_i$. Furthermore, by spreading out, there is a finite extension $\F_q[t_1, \ldots, t_n] \to S$ with $S$ a domain and a finite type $S$-algebra $R''$ such that $R'' \otimes_{S} \Frac(S) \cong R'$; and similarly, a finite $R''$-module $M''$ and ideals $\mathfrak{a}_i''$. Note that $R''$ is then a finite type $\F_q$-algebra. Due to \autoref{pro.shriekfullhomset} (we can always localize so that $\omega_{R''/R'}$ is free), we find elements in $\sC_{M''}$ whose base changes correspond to generators of $\sC$. We denote the algebra generated by those elements by $\sC''$.

By \autoref{le.discretenessfinitefield}, the set 
\[\left\{ \uptau(M'', \sC'', {\fraa''}^{\bm{t}}) \mid  \bm{t} \in [0,T]^n \right\} 
\] is finite. The map $R'' \to R'' \otimes_S \Frac(S)$ is a localization and thus regular. Likewise, since $\el \to \kay$ is separable (and thus regular \cite[\href{https://stacks.math.columbia.edu/tag/07EQ}{Lemma 07EQ}]{stacks-project}), its base change $R' \to R' \otimes_\el \kay$ is also regular by \autoref{co.ffiniteregularbasechange}. Thus \autoref{cor.tauregulartransforms} yields the equality
\[ 
\uptau(M, \sC, \fraa^{\bm{t}}) = \uptau(M'', \sC'', \fraa''^{\bm{t}}) \otimes_{S} \Frac(S) \otimes_\el \kay.
\]
\end{proof}

\subsection{Constancy Regions}

We will essentially follow the general strategy of P\'erez' proof. However, things are a bit more technical since we are not working over regular rings and thus do not have the flatness of Frobenius at our disposal. Consequently, our analog of \cite[Lemma 4.4]{perezconstancyregions} is more involved. We start with a lemma  that is well-known to experts. It is a generalization of the fact that if $t$ is an $F$-jumping number, then so is $pt$.

\begin{lemma}[{\cf \cite[Lemma 4.3]{perezconstancyregions}, \cite[Proposition 3.1]{StablerVfiltrations}, \cite[Lemma 2.1]{StablerAssGraded}}]
\label{le.CartierDividesbyP}
Let $R$ be essentially of finite type over an $F$-finite field. Let $\sC$ be a Cartier algebra that is generated (as an algebra) in degree $e_0$. Let $M$ be a $\sC$-module and $\mathfrak{a}_1, \ldots, \mathfrak{a}_n \subset R$ ideals.
Then \[\uptau(M, \mathfrak{a}_1^{t_1/p^{e_0}}, \cdots, \mathfrak{a}_n^{t_n/p^{e_0}}) = \sC_{e_0} \uptau(M, \mathfrak{a}_1^{t_1}, \cdots, \mathfrak{a}_n^{t_n})\] for any $t \in \mathbb{R}^n_{\geq 0}$.
\end{lemma}
\begin{proof}
By \cite[Theorem 3.4]{BlickleStablerFunctorialTestModules}, we have \[
\uptau(M, \mathfrak{a}_1^{t_1}, \cdots, \mathfrak{a}_n^{t_n}) = \sum_{i=1}^m \sum_{e \geq e_1} \mathcal{C}_e \mathfrak{a}_1^{\lceil t_1 p^e\rceil} \cdots \mathfrak{a}_n^{\lceil t_n p^e\rceil} c_i \underline{H^0_{\eta_i}(M)}
\] for any $e_1 \geq 0$, where the $\eta_1, \cdots, \eta_m$ are the associated primes of $M$ and the $c_i$ form a sequence of test elements in the sense of \cite[Definition 3.1]{BlickleStablerFunctorialTestModules}. We thus have \begin{align*} \sC_{e_0} (\uptau(M,\mathfrak{a}_1^{t_1}, \cdots, \mathfrak{a}_n^{t_n})) &= \sum_{i=1}^m \sum_{e \geq e_1} \sC_{e+e_0} \mathfrak{a}_1^{\lceil t_1 p^e\rceil} \cdots \mathfrak{a}_n^{\lceil t_n p^e\rceil} c_i \underline{H^0_{\eta_i}(M)}  \\&= \sum_{i=1}^m \sum_{e \geq e_0 + e_1} \sC_{e} \mathfrak{a}_1^{\lceil t_1 p^{e-e_0}\rceil} \cdots \mathfrak{a}_n^{\lceil t_n p^{e- e_0}\rceil} c_i \underline{H^0_{\eta_i}(M)}\\ &= \uptau(M, \mathfrak{a}_1^{t_1/p^{e_0}}, \cdots, \mathfrak{a}_n^{t_n/p^{e_0}}), \end{align*} where we use the fact that $\sC$ is generated in degree $e_0$ to obtain the second equality.
\end{proof}

\begin{remark}
Since $M$ is always assumed to be finitely generated, so is $\Hom_R(F_\ast^e M, M)$. In particular, the assumption that $\sC$ is generated in degree $e_0$ implies that we may assume $\sC$ is finitely generated as an algebra.
\end{remark}

\begin{definition}
Let $(M, \sC)$ be a Cartier module, $\mathfrak{a}_1, \ldots, \mathfrak{a}_n \subset R$ be ideals, and $N \subset M$ be an $R$-submodule.
\begin{enumerate}[(a)]
\item Given $\bm{c} \in \mathbb{R}_{\geq 0}^n$, the \emph{constancy region} of $\uptau(M, {\mathfrak{a}}^{\bm{c}})$ is the set of points $\bm{c}' \in \mathbb{R}^n_{\geq 0}$ for which $\uptau(M, {\mathfrak{a}}^{\bm{c}}) = \uptau(M, \mathfrak{a}^{\bm{c}'})$. We denote the corresponding characteristic function by $\rho_{\bm{c}}$, i.e., $\rho_{\bm{c}}(\bm{c}') =1$ if $\uptau(M, {\mathfrak{a}}^{\bm{c}}) = \uptau(M, \mathfrak{a}^{\bm{c}'})$ and zero otherwise.

\item We denote by $B^N({\mathfrak{a}}) \subset \mathbb{R}^n_{\geq 0}$ the set of points $\bm{c} \in \mathbb{R}^n_{\geq 0}$ for which $\uptau(M, {\mathfrak{a}}^{\bm{c}}) \nsubseteq N$. We write $\chi_{{\mathfrak{a}}}^N$ for the corresponding characteristic function.%; i.e., $\chi_{{\mathfrak{a}}}^N(\bm{c}) = 1$ if $\uptau(M, {\mathfrak{a}}^{\bm{c}}) \nsubseteq N$ and zero otherwise. 
\end{enumerate} 
\end{definition}
If $\bm{c} \in B^N(\mathfrak{a})$, the constancy region of $\uptau(M, \fraa^{\bm{c}})$ is contained in $B^N(\mathfrak{a})$. Thus $B^N(\mathfrak{a})$ is a union of constancy regions.

\begin{lemma}
\label{le.constancyregiondescr}
Working in \autoref{set.mixedtestmodules}, for each $\bm{c} \in \mathbb{R}^n_{\geq 0}$ and $T > 0$, there exist finitely many $R$-submodules $N_1, \ldots, N_d, P \subset M$ such that the constancy region of $\uptau(M, \mathfrak{a}^{\bm{c}})$ restricted to $[0,T]^n$ is $\bigcap_{i=1}^d B^{N_i}(\mathfrak{a}) \setminus B^P(\mathfrak{a})$.
\end{lemma}
\begin{proof}
By \autoref{theo.finitenessconstacyregions}, the set $\sA = \{ \uptau(M, \mathfrak{a}^{\bm{c}'}) \, \vert \, \bm{c}' \in [0,T]^n\}$ is finite. Let $N_1, \ldots, N_d$ be the elements of $\sA$ that are strictly contained in $\uptau(M, \mathfrak{a}^{\bm{c}})$ and let $P = \uptau(M, \mathfrak{a}^{\bm{c}})$.
\end{proof}

\begin{remark}
In \autoref{le.constancyregiondescr}, similarly to \cite[Theorem 3.20]{perezconstancyregions}, we do not need to take the intersection of the constancy region with $[0,T]^n$ if the $\fraa_i$ are contained in a single maximal ideal $\fram$. To wit, $\uptau(M, \fraa^{\bm{c}}) \subset \uptau(M, \m^{\bm{c}}) \subset \tau(M, \m^{\sum_i c_i})$. By the Krull intersection theorem, there is $l \gg 0$ such that $\uptau(M, \m^{\sum_i c_i}) \nsubseteq \m^l \uptau(M)$. If $r$ is the number of generators of $\fram$ then, via Skoda, for any $t \geq rl$, we have $\uptau(M, \fram^t) = \fram^l \uptau(M, \fram^{t-l}) \subset \fram^l \uptau(M)$. Hence, if $\bm{c}'$ is in the constancy region of $\uptau(M, \fraa^{\bm{c}})$, then $\sum_i c_i' < r l$; so that the constancy region is bounded.
\end{remark}

\begin{definition}[{see \cite[Definition 4.1]{perezconstancyregions} and \cite[Definition 2.1]{MonskyTeixeiraPFractals1}}]
\label{def.pfractal}
Let $\sF$ be the algebra of functions $\Phi\:\mathbb{R}^n_{\geq 0} \to \mathbb{Q}$. For each $q = p^e$ and $b \in \mathbb{Z}^n$ with $0 \leq b_i \leq q$, we define the operator
$T_{q\vert b}\: \sF \to \sF$ by \[\Phi(t_1, \ldots, t_n) \longmapsto T_{q\vert b} \Phi(t_1, \ldots, t_n) \coloneqq \Phi((t_1 + b_1)/q, \ldots, (t_n + b_n)/q).
\]
We say that $\Phi\colon [0,T]^n \to \mathbb{Q}$ is a \emph{$p$-fractal} if the span of all the $T_{q\vert b} \Phi_0$ is a finite-dimensional $\mathbb{Q}$-vector space, where $\Phi_0$ is the extension by zero of $\Phi$ to $\mathbb{R}^n_{\geq 0}$. We say that $\Phi \in \sF$ is a \emph{$p$-fractal} if the extension by zero of its restriction to each $[0,T]^n$ is a $p$-fractal.
\end{definition}

We now come to our replacement of \cite[Lemma 4.4]{perezconstancyregions}, which, apart from the reduction to the finite field case via \autoref{theo.finitenessconstacyregions}, is the heart of the argument.

\begin{lemma}
\label{le.constancykeylemma}
Working in \autoref{set.mixedtestmodules}, assume that $\sC$ is (finitely) generated in a single degree $e$. Let $\bm{l} = (l_1, \ldots, l_n) \in \mathbb{Z}^n$ be such that $l_i$ is the minimal number of generators of $\mathfrak{a}_i$. Given $T > 0$, there is $c_0$ and a finite subset $\sA$ of $R$-submodules of $M$ such that for all $a \geq 0$, all $\bm{b} \leq p^a$, and all $R$-submodules $N$ of $M$, we have 
\[
T_{p^a\vert \bm{b}} \chi_\mathfrak{a}^N\big\vert_{[0,T]^n} \in \big\{T_{p^{c}\vert \bm{d} } \chi_\mathfrak{a}^{N'}\big\vert_{[0,T]^n} \bigm| c < c_0, \bm{d} \leq p^{c_0} \bm{l}, N' \in \mathcal{A}  \big\}.
\] In other words, $\chi_\mathfrak{a}^N$ is a $p$-fractal.
\end{lemma}
\begin{proof}
If $\uptau(M, \mathfrak{a}^{\bm{t}})$ is generated in gauge $\leq d$, then for any $a_0$ such that $p^{ea_0} > d$, the module $\sC_+^{a_0}(\uptau(M, \mathfrak{a}^{\bm{t}}))$ is generated in gauge $\leq \frac{K}{p-1} +1$, where $K$ is some constant that only depends on $(M, \sC)$, $\mathfrak{a}$, and $T$ (\cite[Corollary 4.10]{BlickleTestIdealsViaAlgebras}). Let $k \coloneqq  \lfloor \frac{K}{p-1} +1\rfloor$. Since $
\dim_{\kay}M_{\leq k} <0$, for any $a \geq a_0$ and any $R$-module $N$, we have $\sC_+^a(\uptau(M, \mathfrak{a}^{\bm{t}})) \nsubseteq N$ if and only if $R\cdot V \nsubseteq N$, where $V$ is some finite-dimensional $\kay$-subspace of $M_{\leq k}$. Set $c_0 = a_0 e$. By the above and \autoref{le.CartierDividesbyP}, for any $a \geq c_0$, we have  
\[
T_{p^a\vert\bm{b}} \chi^N_{\mathfrak{a}}\big\vert_{[0,T]^n} = T_{p^c\vert\bm{b}} \chi^{R \cdot V}_\mathfrak{a}\big\vert_{[0,T]^n}
\] for some $c < c_0$. If the $i$th component $b_i$ of $\bm{b}$ is $\geq l_i p^{c_0}$, Skoda for mixed test modules (\cite[Proposition 4.7]{BlickleStablerFunctorialTestModules}) yields
\[
\uptau(M, \mathfrak{a}^{(\bm{t} + \bm{b})/p^c}) = \mathfrak{a}_i \uptau(M, \mathfrak{a}^{(\bm{t} + \bm{b} - \bm{e}_ip^c)/p^c}), 
\] where $\bm{e}_i$ denotes the $i$th standard basis vector. In particular,
\[
T_{p^c\vert\bm{b}} \chi_\mathfrak{a}^{R \cdot V}\big\vert_{[0,T]^n} = T_{p^c\vert\bm{b} - \bm{e}_ip^c} \chi^{N'}_\mathfrak{a}\big\vert_{[0,T]^n}, 
\]
where $N' \coloneqq (R\cdot V :_M \mathfrak{a}_i)$.

By definition, $\chi_\mathfrak{a}^N$ is the characteristic function of the set $B_\mathfrak{a}^N$, which is a union of constancy regions. By \autoref{theo.finitenessconstacyregions}, there are only finitely many of those. Hence, there are only finitely many functions $\chi_\mathfrak{a}^N$ for varying $N$, and we deduce that we only need to consider $R$-submodules $N'$ contained in a finite set $\sA$.
\end{proof}

\begin{remark}
If we restrict our attention to principal Cartier algebras $\sC = \langle \phi \rangle$, then the assertion of \autoref{le.constancykeylemma} also follows from a more direct argument, avoiding gauge-boundedness. Indeed, one may replace $M$ with $\sC_+^a M$ for $a \gg 0$ and thus assume that $(M, \sC)$ is $F$-pure. Since $\phi\: F_\ast^e M \to M$ is surjective, we have $\phi^a (\uptau(M, \mathfrak{a}^{\bm{t}})) \nsubseteq N$ if and only if $\uptau(M, \mathfrak{a}^{\bm{t}}) \nsubseteq (\phi^{a})^{-1}(N) \subset F^{ea}_\ast M$. In this situation, if one relaxes the definition of the $B_\mathfrak{a}^N$ to only require that the $N$ be subsets (or additive subgroups) of $M$, then the argument of \autoref{le.constancykeylemma} still works.
\end{remark}

\begin{theorem}
\label{theo.pfractal}
Working in \autoref{set.mixedtestmodules}, assume that $\sC$ is (finitely) generated in a single degree. Then, there is a $p$-fractal function $\varphi\: \mathbb{R}_{\geq 0}^n \to \mathbb{N}$ for which
\[\uptau(M, \sC, \mathfrak{a}^{\bm{t}}) = \uptau(M, \sC, \mathfrak{a}^{\bm{t}'}) \Longleftrightarrow \varphi(\bm{t}) = \varphi(\bm{t}').
\]
In particular, the constancy regions are of the form $\varphi^{-1}(i)$ for some $i \in \mathbb{N}$.
\end{theorem}
\begin{proof}
By \autoref{le.constancyregiondescr}, we have $\rho_{\bm{c}} = \chi_\mathfrak{a}^{N_1} \cdots \chi_\mathfrak{a}^{N_d} - \chi_\mathfrak{a}^P$ for some $R$-submodules $N_i, P \subset M$.
By \autoref{theo.finitenessconstacyregions} and Skoda, the set of constancy regions is countable. Fixing an enumeration, we let $\bm{c}_i$ be a point in the $i$th constancy region. Then we can take $\varphi \coloneqq \sum_{i \in \mathbb{N}} i \cdot \rho_{\bm{c}_i}$.
\end{proof}

\subsection{Further questions} We discuss some related questions of general interest.

\begin{question}
Is \autoref{theo.pfractal} true for $\kay$-algebras with $\kay$ an arbitrary $F$-finite field?
\end{question}

If we are in a situation where we reduce $(X, \Delta, \mathfrak{a}_1, \ldots, \mathfrak{a}_n)$, where $K_X + \Delta$ is $\mathbb{Q}$-Cartier, from characteristic zero to characteristic $p$, then it follows from \cite[Theorem 6.8]{HaraYoshidaGeneralizationOfTightClosure} (see also \cite[Proposition 2.11]{BhattSchwedeTakagiWeakOrdinarity}) that the constancy regions of the test ideals converge to the constancy regions of the multiplier ideals as $p \rightarrow \infinity$. Indeed, by \cite{HaraYoshidaGeneralizationOfTightClosure}, we have an inclusion $\uptau(X_p, \Delta_p, \mathfrak{a}_p^{\bm{t}}) \subset \mathcal{J}(X, \Delta, \mathfrak{a}^{\bm{t}})_p$ for all $p$ in a dense open (and $\bm{t}$ arbitrary) and equality for any \emph{fixed} $\bm{t}$ on a dense open. Thus, $\bm{t}$ is a boundary point of the constancy region for $\mathcal{J}$ if and only if it is a boundary point of the constancy regions for $\uptau$ for almost all $p$. A more interesting question to ask is whether the constancy regions of $\uptau$ converge to the constancy regions of $\mathcal{J}$ in a \emph{uniform} way (\cf \cite[Question after 5.5]{perezconstancyregions}). We formulate a precise question below.

Endow $\mathbb{R}^n$ with any metric $d$. Let $M$ be the collection of all non-empty, bounded, and closed subsets of $\mathbb{R}^n$. Then $M$, endowed with the so-called \emph{Hausdorff distance} $d_\mathrm{H}$, becomes a metric space (see e.g. \cite[2.5]{EdgarMeasureTopologyFractalGeometry} for background). Namely, we set 
\[ d_\mathrm{H} \coloneqq \max \Bigg\{ \sup_{a \in A} d(a, B), \sup_{b \in B} d(A, b)\Bigg\}
\] where $d(a, B) \coloneqq \inf_{b \in B} d(a,b)$. Fixing some constancy region $A$ of $\mathcal{J}$, we may fix the corresponding constancy regions $A_p$ in characteristic $p$ (choose points $\bm{t}_p$ that converge pointwise to a point $\bm{t}$ in $A$; then $A_p$ is the constancy region of $\bm{t}$).

\begin{question}
\label{que.hausdorffdistance}
Do the closures of the $A_p$ converge to the closure of $A$ via the Hausdorff-distance, where $d$ is, say, the maximum norm?
\end{question}

\begin{question}
\label{que.hausdorffdimconvergence}
Are there constancy regions $A_p$ in $\mathbb{R}^n$ with Hausdorff dimension $\neq n$? If so, do the Hausdorff dimensions of the $A_p$ converge to the Hausdorff dimension of $A$ as $p \rightarrow \infinity$?
\end{question}

%We expect \autoref{que.hausdorffdimconvergence} to be easier to answer. One might also speculate that a positive answer to the second question, together with pointwise convergence, might be helpful in answering the first one.

\begin{question}
Are the constancy regions of $\uptau$ (path-)connected?
\end{question}

This is already unknown in the case of two ideals. Let $\uptau_1$ and $\uptau_2$ be two connected components ($\subset \mathbb{R}^2$) of the same constancy region, with points $\bm{t} = (t_1, t_2) \in \uptau_1$ and $\bm{t}' =(t'_1, t'_2) \in \uptau_2$. If $t_1 < t'_1$, then necessarily $t_2 > t_2'$. Otherwise, the line between $\bm{t}$ and $\bm{t}'$ would be a non-decreasing path where the test ideals associated with the endpoints coincide. Thus, $\uptau$ would be constant along this line, so that $\bm{t}$ and $\bm{t}'$ belong to the same constancy region.

\begin{definition}
\label{def.mixedfthresholds}
Let $f_1, \ldots, f_n \in R$, and $(M, \sC)$ be an $F$-regular Cartier module. If $(t_1, \ldots, t_{n-1}) \in \mathbb{R}^{n-1}$ is such that $\uptau(M, f_1^{t_1}, \cdots, f_{n-1}^{t_{n-1}}) = M$, the minimal $t_n$ with 
$\uptau(M, \bm{f}^{\bm{t}}) \neq M$ but $\uptau(M, \bm{f}^{\bm{t} - \eps \bm{e}_n}) = M$
is the \emph{$F$-pure threshold}, or we say that $\bm{t}$ is an \emph{$F$-pure threshold}.
\end{definition}

\begin{example}[P\'erez's example]
Here we want to discuss \autoref{que.hausdorffdistance} for the example treated by P\'erez in \cite[Example 5.3]{perezconstancyregions} (where we generalize from $\mathbb{F}_3$ to an $F$-finite field $\kay$ of characteristic $\geq 3$.) We recall the set-up. Let $R = \kay[x,y]$ and $f_1 = x + y$, $f_2 = xy$. Hence, $V(f_1,f_2)$ are three lines in the plane intersecting in the origin.

A blow-up $\pi\: Y \to \mathbb{A}^2_\kay$ of the origin is a log resolution where the total transform of $V(f_1)$ is $F_1 +E$, and that of $V(f_2)$ is $F_2 + 2 E$, where $E$ is the exceptional divisor and $F_i$ is the strict transform. Since the relative canonical divisor is $E$, the mixed multiplier ideal is
\[
\mathcal{J}(\mathbb{A}^2_\kay, f^{\bm{t}}) = \pi_\ast \mathcal{O}_Y(\lceil(1- t_1 - 2t_2)E - t_1 F_1 - t_2 F_2\rceil).
\]
Thus, the constancy regions in the unit square are determined by the line $\el: t_1 + 2t_2 =2$.

P\'erez sketched for $\kay = \mathbb{F}_3$ the $p$-fractal structure of $\uptau(R, f^{\bm{t}})$. Let us briefly explain the heuristic on how to obtain the boundary of the constancy region. First of all, by \autoref{cor.tauregulartransforms}, we may restrict our attention to the case $\kay = \mathbb{F}_p$. Secondly, we describe the boundary region for points of the form $t_1 = a/p \in [0,1]$. As we will see, the $F$-pure threshold is constant in intervals $t_1 \in [a/p, (a+1)/p]$ with $a$ odd. Then we exploit the $p$-fractal structure to show that this behavior propagates to intervals of the form $[a/p^k, (a+1)/p^k]$ with $a$ odd, which are not contained in any of the intervals with smaller $p$-power denominators. Finally, we show that these intervals cover $(0,1)$, thereby obtaining a full description of the boundary region.

We would like to draw the reader's attention to a peculiar behavior for which we have no good heuristic explanation. It seems that if we fix one coordinate, say $t_1 <1$, then the $F$-threshold $t_2$, in the sense of \autoref{def.mixedfthresholds}, avoids $(a/q,a/(q-1))$ for any integer $0 \leq a \leq q
-1$. This phenomenon is well-known in the case $\uptau(f^t)$ (see \cite[Proposition 4.3 (ii)]{BlickleMustataSmithFThresholdsOfHypersurfaces}), but we do not seem to be able to reduce to this case here.

\begin{claim}
\label{claim.Perezexample1}
Fixing $t_1 = a/p<1$ with $a$ even, the $F$-pure threshold of $\uptau(f_1^{t_1}, f_2^{t_2})$ coincides with the log-canonical threshold, i.e., with $t_2 = 1 - \frac{1}{2}t_1$.
\end{claim}
\begin{proof}
Write $a = 2 a'$. The $F$-pure threshold is at most the LCT. Hence, it suffices to show 
\[\uptau\big(f_1^{t_1} f_2^{1 - \frac{1}{2}t_1  - \eps}\big) = R
\] for all $0 < \eps \ll 1$. We may restrict to $\eps = p^{-r}$ for some $r \gg 0$, and then we can compute
\[ 
\uptau\left(f^{t_1}, f^{1 - \frac{1}{2}t_1  - \eps}\right) = \kappa^e\left(\mathbb{F}_p[x,y] \cdot (x+y)^{a p^{e-1}} (xy)^{p^e - a'p^{e-1} - p^{e-r}} \right) ,
\]
for all $e \gg 0$. Since the ideal inside $\kappa^e$ is principal and homogeneous, it suffices to show that its generator admits a summand with exponent $< p^e$ in both $x$ and $y$. Consider the term $ \binom{2a'}{a'} (x^{a'}y^{a'})^{p^e-1} $ in the expression $(x+y)^{a p^{e-1}}$, which is non-zero as $a < p$. Hence,
\[
\uptau\left(f^{t_1}, f^{1 - \frac{1}{2}t_1  - \eps}\right) = \kappa^e \left( \mathbb{F}_p[x,y] \cdot \left(\binom{2a'}{a'} (xy)^{p^e - p^{e-r}} + \text{other terms}\right)\right) = \mathbb{F}_p[x,y].
\]
\end{proof}

\begin{claim}
\label{claim.Perezexample2}
Fixing $t_1 = a/p<1$ with $a$ odd, the $F$-pure threshold of $\uptau(f_1^{t_1}, f_2^{t_2})$ is $t_2 = 1 - \frac{a+1}{2p}$.
\end{claim}
\begin{proof}
Since $p$ is odd, we have $a+1 < p$ even, and by \autoref{claim.Perezexample1}, the $F$-pure threshold at $t_1 = \frac{a+1}{p}$ is $t_2 = 1 - \frac{1}{2}t_1$. Since the $F$-pure threshold cannot increase along an increasing path, we only need to verify that
\[
\uptau\big(f_1^{a/p}, f_2^{1 - (a+1)/2p}\big) \neq \mathbb{F}_p[x,y].
\]
So we compute for $e$ sufficiently large
\begin{align*}
\uptau\Big(f_1^{\frac{a}{p}}, f_2^{1 - \frac{a+1}{2p}}\Big) &= \kappa^e \left( \mathbb{F}_p[x,y] \cdot (x+y)^{a p^{e-1}} (xy)^{p^e - \frac{a+1}{2} p^{e-1}}\right)\\ &= \kappa^e \left( \mathbb{F}_p[x,y] \cdot \sum_{i=0}^a \binom{a}{i} x^{p^e + (i - \frac{a+1}{2}) p^{e-1}} y^{p^e + (a -i - \frac{a+1}{2}) p^{e-1}}\right).
\end{align*} 
In this expression, for each term, the exponent of $x$ or $y$ is $\geq p^e$, and so $\uptau(f_1^{t_1}, f_2^{t_2}) \subset (x,y).$
\end{proof}

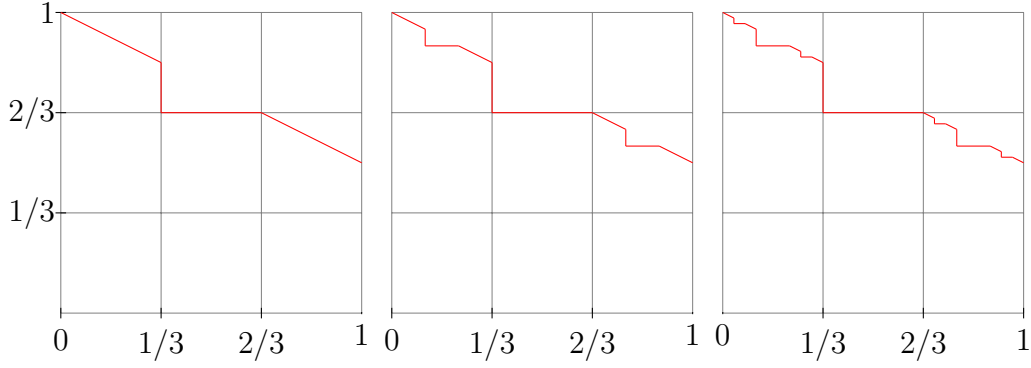
\begin{figure}
%!!! the picture code is scaled by factor of 10!!!
\begin{tikzpicture}[scale=0.4]
\begin{scope}
\draw[step=3.3333,gray,very thin] (0,0)grid(10,10);
\draw(0 cm,5pt)-- (0 cm,-5pt)node[anchor=north] {0};
\draw(3.3333 cm,5pt)-- (3.3333 cm,-5pt)node[anchor=north] {1/3};
\draw(-5pt, 3.3333 cm)-- (5pt,3.3333 cm)node[anchor=east] {1/3};
\draw(6.6666 cm,5pt)-- (6.6666 cm,-5pt)node[anchor=north] {2/3};
\draw(-5pt, 2/3*10 cm)-- (5pt,2/3*10 cm)node[anchor=east] {2/3};
\draw(-5pt, 10 cm)-- (5pt,10 cm)node[anchor=east] {1};
\draw(10 cm,1pt)-- (10 cm,-1pt)node[anchor=north] {1};
\draw[domain=0:3.3333,variable=\x,red, samples=1000] plot({\x}, {10-\x/2});
\draw[domain=3.3333:6.6666,variable=\x,red, samples=1000] plot({\x}, {min(10-\x/2,6.6666)});
\draw[red]  (3.3333,8.3333)-- (3.3333,6.6666);
\draw[domain=6.6666:10,variable=\x,red, samples=1000] plot({\x}, {10-\x/2});
\end{scope}
\begin{scope}[xshift=11cm]
\draw[step=3.3333,gray,very thin] (0,0)grid(10,10);
\draw(0 cm,5pt)-- (0 cm,-5pt)node[anchor=north] {0};
\draw(3.3333 cm,5pt)-- (3.3333 cm,-5pt)node[anchor=north] {1/3};
\draw(6.6666 cm,5pt)-- (6.6666 cm,-5pt)node[anchor=north] {2/3};
\draw(10 cm,1pt)-- (10 cm,-1pt)node[anchor=north] {1};
\draw[domain=0:1.1111,variable=\x,red, samples=1000] plot({\x}, {10-\x/2});
\draw[red] (1.1111, 8.8888) -- (2.2222, 8.8888);
\draw[red] (1.1111, 9.4444) -- (1.1111, 8.8888);
\draw[domain=2.2222:3.3333,variable=\x,red, samples=1000] plot({\x}, {10-\x/2});
\draw[domain=3.3333:6.6666,variable=\x,red, samples=1000] plot({\x}, {min(10-\x/2,6.6666)});
\draw[red]  (3.3333,8.3333)-- (3.3333,6.6666);
\draw[red] (7.7777, 6.1111) -- (7.7777,5.5555);
\draw[red] (7.7777, 5.5555) -- (8.8888,5.5555);
\draw[domain=6.6666:7.7777,variable=\x,red, samples=1000] plot({\x}, {10-\x/2});

\draw[domain=8.8888:10,variable=\x,red, samples=1000] plot({\x}, {10-\x/2});
\end{scope}
\begin{scope}[xshift=22cm]
\draw[step=3.3333,gray,very thin] (0,0)grid(10,10);
\draw(0 cm,1pt)-- (0 cm,-1pt)node[anchor=north] {0};
\draw(3.3333 cm,5pt)-- (3.3333 cm,-5pt)node[anchor=north] {1/3};
\draw(6.6666 cm,5pt)-- (6.6666 cm,-5pt)node[anchor=north] {2/3};
\draw(10 cm,5pt)-- (10 cm,-5pt)node[anchor=north] {1};
\draw[domain=0:0.37037,variable=\x,red, samples=5000] plot({\x}, {10-\x/2});
\draw[red] (10/27, 10 - 10/54) -- (10/27, 10-20/54);
\draw[red] (10/27, 10 - 20/54) -- (20/27, 10-20/54);
\draw[domain=0.7407:1.1111,variable=\x,red, samples=5000] plot({\x}, {10-\x/2});
\draw[red] (1.1111, 8.8888) -- (2.2222, 8.8888);
\draw[red] (1.1111, 9.4444) -- (1.1111, 8.8888);
\draw[domain=10*8/27:3.3333,variable=\x,red, samples=5000] plot({\x}, {10-\x/2});
\draw[domain=10*6/27:10*7/27,variable=\x,red, samples=5000] plot({\x}, {10-\x/2});
\draw[red] (70/27, 10 - 70/54) -- (70/27, 10-80/54);
\draw[red] (70/27, 10 - 80/54) -- (80/27, 10-80/54);
% second and last box
\draw[domain=3.3333:6.6666,variable=\x,red, samples=5000] plot({\x}, {min(10-\x/2,6.6666)});
\draw[red]  (3.3333,8.3333)-- (3.3333,6.6666);
\draw[red] (7.7777, 6.1111) -- (7.7777,5.5555);
\draw[red] (7.7777, 5.5555) -- (8.8888,5.5555);
\draw[red](190/27, 10 - 190/54) -- (190/27, 10 - 200/54);
\draw[red](190/27, 10 - 200/54) -- (200/27, 10 - 200/54);
\draw[domain=200/27:210/27,variable=\x,red, samples=5000] plot({\x}, {10-\x/2});
\draw[domain=180/27:190/27,variable=\x,red, samples=5000] plot({\x}, {10-\x/2});
\draw[domain=260/27:10,variable=\x,red, samples=5000] plot({\x}, {10-\x/2});
\draw[domain=240/27:250/27,variable=\x,red, samples=5000] plot({\x}, {10-\x/2});
\draw[red] (250/27, 10 - 1/2*250/27) -- (250/27, 10 - 1/2*250/27 - 5/27);
\draw[red] (250/27, 10 - 1/2*250/27 - 5/27) -- (260/27,10 - 1/2*250/27 - 5/27);
\end{scope}
\end{tikzpicture}
\caption{A recursive way to obtain the boundary for $p = 3$. The graph on the left is scaled by $1/3$ and replaces the graph in the intervals $[0,1/3]$ and $[2/3,1]$. The so obtained graph is then scaled by $1/3$ and replaces the graph in the intervals $[0,1/3]$ and $[2/3,1]$ in the final picture. Iterating this process indefinitely gives the boundary region.}
\end{figure}

\begin{claim}
Let $0 \leq l \leq \frac{p-1}{2}$. Then $T_{p\vert (2l, p -l-1)} \chi_{f}^{(x,y)} = \chi_{f}^{(x,y)}$. In particular, the points \begin{align*} &T_{p\vert (2l_{k-1}, p-l_{k-1}-1)}\circ \cdots \circ T_{p\vert (2l_1, p-l_1-1)}\left(\frac{a}{p}, 1-  \frac{a+1}{2p}\right) \text{ and }\\ &T_{p\vert (2l_{k-1}, p-l_{k-1}-1)} \circ \cdots \circ T_{p\vert (2l_1, p-l_1-1)} \left(\frac{a+1}{p}, 1-  \frac{a+1}{2p} \right)
\end{align*}
are boundary points of $\chi_{f}^{(x,y)}$ for all $k$, $0 \leq l_i \leq \frac{p-1}{2}$, and $a$ odd.
\end{claim}
\begin{proof}
We use \cite[Lemma 4.4]{perezconstancyregions} to compute $
T_{p\vert (2l, p -l-1)} \chi_{f}^{(x,y)} = \chi_{f}^I$,  where
\[
I = \left(\left((x^p,y^p): (xy)^{p-l-1}\right) : (x+y)^{2l} \right) = \left((x^{l+1},y^{l+1}) : (x+y)^{2l} \right).
\] For $l = 0$, we have that $I = (x,y)$. So assume that $l > 0$; then 
\[ (x+y)^{2l} = \sum_{i=0}^{2l} \binom{2l}{i} x^{2l -i}y^{i} = \binom{2l}{l} x^l y^l + (x^{l+1},y^{l+1}).
\] The binomial coefficient is non-zero as $l \neq 0$. From this, we conclude that $I = (x,y)$.

For the addendum, we proceed as follows. By definition, a point $v$ lies on the boundary of a constancy region $C$ if for all $\varepsilon >0$, any ball centered around $v$ with radius $\varepsilon$ contains points lying in $C$ and points not lying in $C$. The points are obtained from \autoref{claim.Perezexample1} and \autoref{claim.Perezexample2} by applying the appropriate transformations $T_{p\vert(2l, p-l-1)}$ $k-1$ times. Since the points in \autoref{claim.Perezexample1} and \autoref{claim.Perezexample2} lie on the boundary, $T_{p\vert(2 l, p - l -1)}(v - \varepsilon w) = T_{p\vert(2 l, p - l -1)}(v) - \varepsilon/p \cdot w$ and $T_{p\vert(2 l, p - l -1)} \chi_f^{(x,y)} = \chi_f^{(x,y)}$, we conclude that $T_{p\vert(2 l, p - l -1)}(v)$ is also a boundary point.
\end{proof}

\begin{claim}
The points 
\[
T_{p\vert (2l_{k-1}, p-l_{k-1}-1)}\circ \cdots \circ T_{p\vert (2l_1, p-l_1-1)}\left(a/p, 1- (a+1)/2p\right),
\] 
with $0 \leq l_i \leq \frac{p-1}{2}$ and $0 < a < p$ odd, are the boundary points whose first coordinate is of the form $b/p^k$ with $b$ odd, $0 < b <p^k$, and not contained in any interval $[a/p^r,(a+1)/p^r]$ with $0 < a < p^r$ odd and $r < k$. The second coordinate of such a boundary point is $1 - \frac{1}{2} \cdot \frac{b+1}{p^k}$
\end{claim}
\begin{proof}
Writing the base-$p$ terminating expansion $b/p^k=0.b_1b_2\ldots b_k$, we see that it is not contained in such an interval if and only if all $b_1, \ldots, b_{k-1}$ are even. By our assumption on $b$, we must have that $b_k$ is odd. There are precisely $(\frac{p+1}{2})^{k-1}\cdot\frac{p-1}{2} $ points of this type. This is also the number of points produced by applying $T_{p\vert (2l_i, p-l_i-1)}$ $k-1$ times, so it suffices to show one inclusion. To this end, it suffices to show that if $a/p^k$, $a$ odd, and $0 < a < p^k$ is not contained in any interval $[b/p^r,(b+1)p^r]$ for $r < k$, and $0 < b <p^r$ odd, then 
\[ \frac{a + 2p^kl_i}{p^{k+1}} \notin  \left[\frac{b}{p^r},\frac{b+1}{p^r}\right]
\] for any $0 < b < p^r$ odd and $r \leq k$. By the above, the (terminating) base-$p$ expansion of $a/p^k$ has the form $0.a_1\ldots a_k$ with $a_1, \ldots, a_{k-1}$ even and $a_k$ odd. We divide $a$ by $p$, which shifts indices one to the right. Since we then add $\frac{2l_i}{p}$, the first term in the base $p$ expansion is even.

In order to compute the second coordinate of such a boundary point, it suffices to show by induction that if we apply $T_{p\vert(2l, p-l-1)}$ to a point $(b/p^k, 1 - (b+1)/2p^k)$, then the second coordinate is of the form $1 - \frac{1}{2} \frac{b+1 + 2p^k l}{2p^{k+1}}$. This in turn follows from a direct computation.
\end{proof}

\begin{claim}
The intervals $[a/p^k,(a+1)/p^k]$ with $a$ odd, $0 < a < p^k$ and $k \in \mathbb{N}$ cover $(0,1)$.
\end{claim}
\begin{proof}
Given $\frac{c}{d}$ with $c < d$, we have to show that there is $a$ and $k$ as above such that $ {a}/{p^k} \leq {c}/{d} \leq (a+1)/{p^k}$. Choose any $a, k$.  We may assume that $ {(a-1)}/{p^k} \leq {c}/{d} \leq a/{p^k}$. Then, we may choose $m \gg 0$ such that 
\[ \frac{(a-1) p^{m-k} +1}{p^{m}} = \frac{a-1}{p^k} + \frac{1}{p^m} < \frac{c}{d} < \frac{a}{p^k} + \frac{1}{p^m} = \frac{a p^{m-k} +1}{p^m}. \] Hence, $\frac{c}{d}$ is in the interval $[b/p^m,(b+1)/{p^m}]$ with $b = (a-1) p^{m-k} +1$, which is odd.
\end{proof}

\begin{claim}
If $A_p$ and $A$ denote the first constancy region in characteristic $p$ and zero, respectively, then $d_\mathrm{H}(A, A_p)=1/p$, where we take the maximum norm as our metric on $\mathbb{R}^2$.
\end{claim}
\begin{proof}
If the first coordinate of $\bm{t} \in A_p$ is in an interval $[{a}/{p^k}, (a+1){p^k}]$ with $a$ odd, then $d(\bm{t}, A) \leq {1}/{p^k}$, and equality is attained if $t_1 = {a}/{p^k}$ avoids any interval with a smaller exponent. Since $A_p \subset A$, we conclude that $d_\mathrm{H}(A_p, A) = {1}/{p}$.
\end{proof}

Finally, we note that the length of the boundary of $A_p$ is $2$. Indeed, the contribution by the staircase function is given by
\[ \sum_{k=1}^\infinity \frac{3}{2} \frac{1}{p^k} \cdot \left(\frac{p+1}{2}\right)^{k-1} \cdot \frac{p-1}{2} = \frac{3}{2} \frac{p-1}{p+1} \left(\sum_{k=0}^\infinity \left(\frac{p+1}{2p} \right)^k -1\right)=  \frac{3}{2}.\]
Here $\frac{3}{2} \frac{1}{p^k}$ is the length of any staircase segment. Indeed, for $k =1$ the length of such a staircase segment is given by \[ \left\lVert \biggl(\frac{a}{p},1 - \frac{a+1}{2p}\biggr) - \biggl(\frac{a}{p}, 1 - \frac{1}{2}\cdot \frac{a}{p}\biggr) \right\rVert + \left\lVert\biggl(\frac{a}{p}, 1 - \frac{a+1}{2p}\biggr) - \biggl(\frac{a+1}{p}, 1 - \frac{1}{2}\frac{a+1}{p}\biggr) \right\rVert = \frac{1}{2p} + \frac{1}{p}.\] Since the operators $T_{p\vert -}$ scale everything by $1/p$ and translate, the length of a staircase segment obtained by iterating $T$ $k$ times is as claimed. The last two terms in the product are the number of these segments. Since the other boundary parts (given by parts of the coordinate axes and the line segment between $(1,0)$ and $(1, 1/2)$) have length $\frac{5}{2}$, we get $2$ in total. In particular, the Hausdorff dimension of the boundary is $1$.
\end{example}

\bibliographystyle{skalpha}
\bibliography{MainBib}

\end{document}